%% file: cr_reduced.tex
\documentclass{amsart}
\input{header}
\usepackage{esint}

\begin{document}

\title{A Paneitz-type operator for CR pluriharmonic functions}
\author{Jeffrey S.\ Case}
\thanks{JSC was partially supported by NSF Grant No.\ DMS-1004394}
\address{Department of Mathematics \\ Princeton University \\ Princeton, NJ 08540}
\email{jscase@math.princeton.edu}
\author{Paul Yang}
\thanks{PY was partially supported by NSF Grant No.\ DMS-1104536}
\address{Department of Mathematics \\ Princeton University \\ Princeton, NJ 08540}
\email{yang@math.princeton.edu}
\begin{abstract}
We introduce a fourth order CR invariant operator on pluriharmonic functions on a three-dimensional CR manifold, generalizing to the abstract setting the operator discovered by Branson, Fontana and Morpurgo.  For a distinguished class of contact forms, all of which have vanishing Hirachi-$Q$ curvature, these operators determine a new scalar invariant with properties analogous to the usual $Q$-curvature.  We discuss how these are similar to the (conformal) Paneitz operator and $Q$-curvature of a four-manifold, and describe its relation to some problems for three-dimensional CR manifolds.  
\end{abstract}
\maketitle


\input{intro}


\input{bg}                   
\input{lee}                  
\input{defn}                 
\input{covariance}           
\input{qprime}               

\appendix
\input{tractor}              
\input{fefferman}            

\bibliographystyle{abbrv}
\bibliography{../bib}
\end{document}

%% file: header.tex
\usepackage{amsmath}
\usepackage{amssymb}
\usepackage{amsthm}

\DeclareMathOperator{\Vol}{Vol}

\DeclareMathOperator{\Imaginary}{Im}
\DeclareMathOperator{\Real}{Re}

\DeclareMathOperator{\contr}{\lrcorner}
\newcommand{\dbar}{\overline{\partial}}

\newcommand{\lp}{\langle}
\newcommand{\rp}{\rangle}
\newcommand{\lv}{\lvert}
\newcommand{\rv}{\rvert}

\newcommand{\on}[2]{\mathop{\null#2}\limits^{#1}}
\newcommand{\tracefree}[1]{\on{\circ}{#1}}

\newcommand{\mE}{\mathcal{E}}

\newcommand{\mP}{\mathcal{P}}

\newcommand{\bC}{\mathbb{C}}
\newcommand{\bD}{\mathbb{D}}

\newcommand{\bR}{\mathbb{R}}

\newcommand{\bZ}{\mathbb{Z}}

\newcommand{\comment}[1]{}

\newtheorem{thm}{Theorem}[section]
\newtheorem{prop}[thm]{Proposition}
\newtheorem{lem}[thm]{Lemma}

\theoremstyle{definition}
\newtheorem{defn}[thm]{Definition}

\theoremstyle{remark}
\newtheorem{remark}[thm]{Remark}

\numberwithin{equation}{section}

%% file: intro.tex
\section{Introduction}
\label{sec:intro}

It is well-known that there is a deep analogy between the study of three-dimensional CR manifolds and of four-dimensional conformal manifolds.  Two important ingredients in the study of the latter are the Paneitz operator $P_4$ and the $Q$-curvature $Q_4$.  Given a metric $g$, the Paneitz operator is a formally self-adjoint fourth-order differential operator of the form $\Delta^2$ plus lower-order terms, while the $Q$-curvature is a scalar invariant of the form $\Delta R$ plus lower-order terms, where $R$ is the scalar curvature of $g$ and ``order'' is measured according to the number of derivatives taken of $g$.  The pair $(P_4,Q_4)$ generalizes to four-dimensions many important properties of the pair $(-\Delta,K)$ of the Laplacian and the Gauss curvature of a two-manifold.  For example, if $(M^4,g)$ is a Riemannian manifold and $\hat g=e^{2\sigma}g$ is another choice of metric, then
\begin{align}
\label{eqn:conformal_p_trans} e^{4\sigma}\hat P_4(f) & = P_4(f) \\
\label{eqn:conformal_q_trans} e^{4\sigma}\hat Q_4 & = Q_4 + P_4(\sigma)
\end{align}
for all $f\in C^\infty(M)$.  Since also $P_4(1)=0$, the transformation formula~\eqref{eqn:conformal_q_trans} implies that on a compact conformal manifold $(M^4,[g])$, the integral of the $Q$-curvature is a conformal invariant; indeed, the Gauss--Bonnet--Chern formula states that this integral is a linear combination of the Euler characteristic of $M^4$ and the integral of a pointwise conformal invariant, namely the norm of the Weyl tensor.  The pair $(P_4,Q_4)$ also appears in the linearization of the Moser--Trudinger inequality.  Denoting by $(S^4,g_0)$ the standard four-sphere with $g_0$ a metric of constant sectional curvature one, it was proven by Beckner~\cite{Beckner1993}, and later by Chang and the second author~\cite{ChangYang1995} using a different technique, that
\begin{equation}
\label{eqn:conformal_mt}
\int_{S^4} u\,P_4u + 2\int_{S^4} Q_4u - \frac{1}{2}\left(\int_{S^4} Q_4\right)\log\left(\fint_{S^4} e^{4u}\right) \geq 0
\end{equation}
for all $u\in C^\infty(S^4)$, and that equality holds if and only if $e^{2u}g_0$ is an Einstein metric on $(S^4,g_0)$.

A natural question is whether there exist analogues of $P_4$ and $Q_4$ defined for a three-dimensional pseudohermitian manifold $(M^3,J,\theta)$.  In a certain sense this is already known; the compatibility operator studied by Graham and Lee~\cite{GrahamLee1988} is a fourth-order CR invariant operator with leading order term $\Delta_b^2+T^2$ and Hirachi~\cite{Hirachi1990} has identified a scalar invariant $Q_4$ which is related to $P_4$ through a change of contact form in a manner analogous to~\eqref{eqn:conformal_q_trans}.  However, while the total $Q$-curvature of a compact three-dimensional CR manifold is indeed a CR invariant, it is always equal to zero.  Moreover, the $Q$-curvature of the standard CR three-sphere vanishes identically; indeed, this is true for the boundary of any strictly pseudoconvex domain~\cite{FeffermanHirachi2003}, as is explained in Section~\ref{sec:defn}.  In particular, while~\eqref{eqn:conformal_mt} is true on the CR three-sphere, it is trivial, as it only states that the Paneitz operator is nonnegative.

Using spectral methods, Branson, Fontana and Morpurgo~\cite{BransonFontanaMorpurgo2007} have recently identified a new operator $P_4^\prime$ on the standard CR three-sphere $(S^3,J,\theta_0)$ such that $P_4^\prime$ is of the form $\Delta_b^2$ plus lower-order terms, $P_4^\prime$ is invariant under the action of the CR automorphism group of $S^3$, and $P_4^\prime$ appears in an analogue of~\eqref{eqn:conformal_mt} in which the exponential term is present.  There is, however, a catch: the operator $P_4^\prime$ acts only on the space $\mP$ of CR pluriharmonic functions on $S^3$, namely those functions which are the boundary values of pluriharmonic functions in the ball $\{(z,w)\colon \lv z\rv^2+\lv w\rv^2<1\}\subset\bC^2$.  The space of CR pluriharmonic functions on $S^3$ is itself invariant under the action of the CR automorphism group, so it makes sense to discuss the invariance of $P_4^\prime$.  Using this operator, Branson, Fontana and Morpurgo~\cite{BransonFontanaMorpurgo2007} showed that
\begin{equation}
\label{eqn:cr_mt}
\int_{S^3} u\,P_4^\prime u + 2\int_{S^3} Q_4^\prime u - \left(\int_{S^3} Q_4^\prime\right)\log\left(\fint_{S^3} e^{2u}\right) \geq 0
\end{equation}
for all $u\in\mP$, where $Q_4^\prime=1$ and equality holds in~\eqref{eqn:cr_mt} if and only if $e^u\theta_0$ is a torsion-free contact form with constant Webster scalar curvature.

Formally, the operator $P_4^\prime$ is constructed using Branson's principle of analytic continuation in the dimension~\cite{Branson1995}.  More precisely, there exists in general dimensions a fourth-order CR invariant operator with leading order term $\Delta_b^2+T^2$, which we shall also refer to as the Paneitz operator.  On the CR spheres, this is an intertwining operator, and techniques from representation theory allow one to quickly compute the spectrum of this operator.  By carrying out this program, one observes that the Paneitz operator on the standard CR three-sphere kills CR pluriharmonic functions, and moreover, the Paneitz operator $P_{4,n}$ on the standard CR $(2n+1)$-sphere acts on CR pluriharmonic functions as $\frac{n-1}{2}$ times a well-defined operator, called $P_4^\prime$.  One observation in~\cite{BransonFontanaMorpurgo2007} is that this operator is in fact a fourth-order differential operator acting on CR pluriharmonic functions which is, in a suitable sense, CR invariant.

The purpose of this article is to show that there is a meaningful definition of the ``$P^\prime$-operator'' on general three-dimensional CR manifolds enjoying the same algebraic properties as the operator $P_4^\prime$ defined in~\cite{BransonFontanaMorpurgo2007}, and also to investigate the possibility of defining a scalar invariant $Q_4^\prime$ which is related to $P_4^\prime$ in a manner analogous to the way in which the $Q$-curvature is related to the Paneitz operator.  It turns out that one cannot define $Q_4^\prime$ in a meaningful way for a general choice of contact form on a CR three-manifold, though one can for a distinguished class of contact forms, namely the so-called pseudo-Einstein contact forms.  These are precisely those contact forms which are locally volume-normalized with respect to a closed section of the canonical bundle, which is a meaningful consideration in dimension three (cf.\ \cite{Lee1988} and Section~\ref{sec:lee}).  Having made these definitions, we will also begin to investigate the geometric meaning of these invariants.

To describe our results, let us begin by discussing in more detail the ideas which give rise to the definitions of $P_4^\prime$ and $Q_4^\prime$.  To define $P_4^\prime$, we follow the same strategy of Branson, Fontana, and Morpurgo~\cite{BransonFontanaMorpurgo2007}.  First, Gover and Graham~\cite{GoverGraham2005} have shown that on a general CR manifold $(M^{2n+1},J)$, one can associate to each choice of contact form $\theta$ a formally-self adjoint real fourth-order operator $P_{4,n}$ which has leading order term $\Delta_b^2+T^2$, and that this operator is CR covariant.  On three-dimensional CR manifolds, this reduces to the well-known operator
\[ P_4 := P_{4,1} = \Delta_b^2 + T^2 - 4\Imaginary \nabla^\alpha A_{\alpha\beta}\nabla^\beta \]
which, through the work of Graham and Lee~\cite{GrahamLee1988} and Hirachi~\cite{Hirachi1990}, is known to serve as a good analogue of the Paneitz operator of a four-dimensional conformal manifold.  As pointed out by Graham and Lee~\cite{GrahamLee1988}, the kernel of $P_4$ (as an operator on a three-dimensional CR manifold) contains the space $\mP$ of CR pluriharmonic functions, and thus one can ask whether the operator
\[ P_4^\prime := \lim_{n\to1} \frac{2}{n-1}P_{4,n} \rv_{\mP} \]
is well-defined.  As we verify in Section~\ref{sec:defn}, this is the case.  It then follows from standard arguments (cf.\ \cite{BransonGover2005}) that if $\hat\theta=e^\sigma\theta$ is any other choice of contact form, then the corresponding operator $\widehat{P_4^\prime}$ is related to $P_4^\prime$ by
\begin{equation}
\label{eqn:cr_p_trans}
e^{2\sigma}\widehat{P_4^\prime}(f) = P_4^\prime(f) + P_4(\sigma f)
\end{equation}
for any $f\in\mP$.  Thus the relation between $P_4^\prime$ and $P_4$ is analogous to the relation~\eqref{eqn:conformal_q_trans} between the $Q$-curvature and the Paneitz operator; more precisely, the $P^\prime$-operator can be regarded as a $Q$-curvature operator in the sense of Branson and Gover~\cite{BransonGover2005}.  Moreover, since the Paneitz operator is self-adjoint and kills pluriharmonic functions, the transformation formula~\eqref{eqn:cr_p_trans} implies that
\[ e^{2\sigma}\widehat{P_4^\prime}(f) = P_4^\prime(f) \mod \mP^\perp \]
for any $f\in\mP$, returning $P_4^\prime$ to the status of a Paneitz-type operator.  This is the sense in which the $P^\prime$-operator is CR invariant, and is the way that it is studied in~\eqref{eqn:cr_mt}.

From its construction, one easily sees that $P_4^\prime(1)$ is exactly Hirachi's $Q$-curvature.  Thus, unlike the Paneitz operator, the $P^\prime$-operator does not necessarily kill constants.  However, there is a large and natural class of contact forms for which the $P^\prime$-operator does kill constants, namely the pseudo-Einstein contact forms; see Section~\ref{sec:lee} for their definition.  It turns out that two pseudo-Einstein contact forms $\hat\theta$ and $\theta$ must be related by a CR pluriharmonic function, $\log\hat\theta/\theta\in\mP$ (cf.\ \cite{Lee1988}).  If $(M^3,J)$ is the boundary of a domain in $\bC^2$, such contact forms exist in profusion, arising as solutions to Fefferman's Monge-Amp\`ere equation (cf.\ \cite{Fefferman1976,FeffermanHirachi2003}).  In this setting, it is natural to ask whether there is a scalar invariant $Q_4^\prime$ such that $P_4^\prime(1)=\frac{n-1}{2}Q_4^\prime$.  This is true; we will show that if $(M^3,J,\theta)$ is a pseudo-Einstein manifold, then the scalar invariant
\[ Q_4^\prime := \lim_{n\to1} \frac{4}{(n-1)^2}P_{4,n}(1) \]
is well-defined.  As a consequence, if $\hat\theta=e^\sigma\theta$ is another pseudo-Einstein contact form (in particular, $\sigma\in\mP$), then
\begin{equation}
\label{eqn:cr_q_trans}
e^{2\sigma}\widehat{Q_4^\prime} = Q_4^\prime + P_4^\prime(\sigma) + \frac{1}{2}P_4(\sigma^2) .
\end{equation}
Taking the point of view that $P_4^\prime$ is a Paneitz-type operator, we may also write
\[ e^{2\sigma}\widehat{Q_4^\prime} = Q_4^\prime + P_4^\prime(\sigma) \mod \mP^\perp . \]
The upshot is that, on the standard CR three-sphere, $Q_4^\prime=1$, so that this indeed recovers the interpretation of the Beckner--Onofri-type inequality~\eqref{eqn:cr_mt} of Branson--Fontana--Morpurgo~\cite{BransonFontanaMorpurgo2007} as an estimate involving some sort of Paneitz-type operator and $Q$-type curvature.  Additionally, we also see from~\eqref{eqn:cr_q_trans} that the integral of $Q_4^\prime$ is a CR invariant; more precisely, if $(M^3,J)$ is a compact CR three-manifold and $\theta,\hat\theta$ are two pseudo-Einstein contact forms, then
\[ \int_M \widehat{Q_4^\prime}\,\hat\theta\wedge d\hat\theta = \int_M Q_4^\prime\, \theta\wedge d\theta . \]

In conformal geometry, the total $Q$-curvature plays an important role in controlling the topology of the underlying manifold.  For instance, the total $Q$-curvature can be used to prove sphere theorems (e.g.\ \cite[Theorem~B]{Gursky1999} and~\cite[Theorem~A]{ChangGurskyYang2003}).  We will prove the following CR analogue of Gursky's theorem~\cite[Theorem~B]{Gursky1999}.

\begin{thm}
\label{thm:qprime_upper_bound}
Let $(M^3,J,\theta)$ be a compact three-dimensional pseudo-Einstein manifold with nonnegative Paneitz operator and nonnegative CR Yamabe constant.  Then
\[ \int_M Q_4^\prime \, \theta\wedge d\theta \leq \int_{S^3} Q_0^\prime\,\theta_0\wedge d\theta_0, \]
with equality if and only if $(M^3,J)$ is CR equivalent to the standard CR three sphere.
\end{thm}

Here, the CR Yamabe constant of a CR manifold $(M^3,J)$ is the infimum of the total Webster scalar curvature over all contact forms $\theta$ such that $\int\theta\wedge d\theta=1$ (cf.\ \cite{JerisonLee1987}).  The proof of Theorem~\ref{thm:qprime_upper_bound} relies upon the existence of a CR Yamabe contact form --- that is, the existence of a smooth unit-volume contact form with constant Webster scalar curvature equal to the CR Yamabe constant~\cite{ChengMalchiodiYang2013,JerisonLee1987}.  In particular, it relies on the CR Positive Mass Theorem~\cite{ChengMalchiodiYang2013}.  One complication which does not arise in the conformal case~\cite{Gursky1999} is the possibility that the CR Yamabe contact form may not be pseudo-Einstein.  We overcome this difficulty by computing how the local formula~\eqref{eqn:q4prime_crit} for $Q_4^\prime$ transforms with a general change of contact form; i.e.\ without imposing the pseudo-Einstein assumption.  For details, see Section~\ref{sec:qprime}.

In conformal geometry, the total $Q$-curvature also arises when considering the Euler characteristic of the underlying manifold.  Burns and Epstein~\cite{BurnsEpstein1988} have shown that there is a biholomorphic invariant, now known as the Burns--Epstein invariant, of the boundary of a strictly pseudoconvex domain which is related to the Euler characteristic of the domain in a similar way.  It turns out that the Burns--Epstein invariant is a constant multiple of the total $Q^\prime$-curvature, and thus there is a nice relationship between the total $Q^\prime$-curvature and the Euler characteristic.

\begin{thm}
\label{thm:burns_epstein}
Let $(M^3,J)$ be a compact CR manifold which admits a pseudo-Einstein contact form $\theta$, and denote by $\mu(M)$ the Burns--Epstein invariant of $(M^3,J)$.  Then
\[ \mu(M) = -16\pi^2\int_M Q^\prime\,\theta\wedge d\theta . \]
In particular, if $(M^3,J)$ is the boundary of a strictly pseudoconvex domain $X$, then
\[ \int_X\left(c_2 - \frac{1}{3}c_1^2\right) = \chi(X) - \frac{1}{16\pi^2}\int_M Q^\prime\,\theta\wedge d\theta , \]
where $c_1$ and $c_2$ are the first and second Chern forms of the K\"ahler--Einstein metric in $X$ obtained by solving Fefferman's equation and $\chi(X)$ is the Euler characteristic of $X$.
\end{thm}

While we were discussing a preliminary version of this work at Banff in Summer 2012, it was suggested to us by Kengo Hirachi that a version of Theorem~\ref{thm:burns_epstein} should be true.  It was then pointed out to us by Jih-Hsin Cheng that Theorem~\ref{thm:burns_epstein} can be proved by using the formula given by Burns and Epstein~\cite{BurnsEpstein1988} (see also~\cite{ChengLee1990}) for their invariant.  This fact has since been independently verified by Hirachi~\cite{Hirachi2013}, to which we refer the reader for the details of the verification of Theorem~\ref{thm:burns_epstein}.

Finally, we point out that much of the background described above generalizes to higher dimensions.  On any even-dimensional Riemannian manifold $(M^{2n},g)$ there exists a pair $(P_{2n},Q_{2n})$ of a conformally-invariant differential operator $P_{2n}$ of the form $(-\Delta)^n$ plus lower order terms, the so-called GJMS operators~\cite{GJMS1992}, and scalar invariants $Q_{2n}$ of the form $(-\Delta)^{n-1}R$ plus lower-order terms, the so-called (critical) $Q$-curvatures~\cite{Branson1995}, which satisfy transformation rules analogous to~\eqref{eqn:conformal_p_trans} and~\eqref{eqn:conformal_q_trans}.  On the standard $2n$-sphere, Beckner~\cite{Beckner1993} and Chang--Yang~\cite{ChangYang1995} showed that the analogue of~\eqref{eqn:conformal_mt} still holds, including the characterization of equality.  Likewise, Branson, Fontana and Morpurgo~\cite{BransonFontanaMorpurgo2007} defined operators $P_{2n+2}^\prime$ on the standard CR $(2n+1)$-sphere which are CR invariant operators of order $2n+2$ and for which an analogue of~\eqref{eqn:cr_mt} holds, including the characterization of equality, where again $Q_{2n+2}^\prime$ are only identified as explicit constants.  After a preliminary version of this article was presented at Banff in Summer 2012, Hirachi~\cite{Hirachi2013} showed how to use the ambient calculus to extend the $P^\prime$-curvature and $Q^\prime$-curvature to higher dimensions in such a way that the transformation formulae~\eqref{eqn:cr_p_trans} and~\eqref{eqn:cr_q_trans} hold.  In a forthcoming work with Rod Gover, we produce tractor formulae for the $P^\prime$-operator and the $Q^\prime$-curvature.  This allows us to produce for pseudo-Einstein manifolds with vanishing torsion a product formula for the $P^\prime$-operator and an explicit formula for the $Q^\prime$-curvature, giving a geometric derivation of the formulae given by Branson, Fontana and Morpurgo~\cite{BransonFontanaMorpurgo2007}.

This article is organized as follows.  In Section~\ref{sec:bg}, we recall some basic definitions and facts in CR geometry, and in particular recall the depth of the analogy between aspects of conformal and CR geometry.  In Section~\ref{sec:lee}, we introduce the notion of a pseudo-Einstein contact form on a three-dimensional CR manifold, and explore some basic properties of such forms.  In Section~\ref{sec:defn}, we give a general formula for the Paneitz operator on a CR manifold $(M^{2n+1},J,\theta)$.  We then use this formula to give the definitions of the $P^\prime$-operator and the $Q^\prime$-curvature, and establish some of their basic properties.  In Section~\ref{sec:covariance}, we check by direct computation that the $P^\prime$-operator satisfies the correct transformation law.  Indeed, this computation shows that $P^\prime$ no longer satisfies this rule if it is considered on a space strictly larger than the space of CR pluriharmonic functions.  In Section~\ref{sec:qprime}, we check by direct computation that the $Q^\prime$-curvature satisfies the correct transformation law, and use this computation to prove Theorem~\ref{thm:qprime_upper_bound}.  In the appendices, we will derive in two different ways the local formula for the CR Paneitz operator in general dimension.  First, Appendix~\ref{sec:tractor} gives the derivation using the CR tractor calculus~\cite{GoverGraham2005}.  Second, Appendix~\ref{sec:fefferman} gives the derivation using Lee's construction~\cite{Lee1986} of the Fefferman bundle.

%% file: bg.tex
\section{CR Geometry}
\label{sec:bg}

Throughout this article, we will follow the conventions used by Gover and Graham~\cite{GoverGraham2005} for describing CR and pseudohermitian invariants and performing local computations using a choice of contact form.  These conventions are identical to the the conventions used by Lee in his work on pseudo-Einstein structures~\cite{Lee1988}, except that we will sometimes describe invariants as densities rather than functions.  This has the effect that exponential factors will generally not appear in our formulae for how these invariants transform under a change of contact form.  Both for the convenience of the reader and to hopefully avoid any confusion caused by the many different notations used in the literature, we use this section to make precise these conventions as necessary for this article.

\subsection{CR and pseudohermitian manifolds}
\label{sec:bg/cr}

A \emph{CR manifold} is a pair $(M^{2n+1},J)$ of a smooth oriented (real) $(2n+1)$-dimensional manifold together with a formally integrable complex structure $J\colon H\to H$ on a maximally nonintegrable codimension one subbundle $H\subset TM$.  In particular, the bundle $E=H^\perp\subset T^\ast M$ is orientable and any nonvanishing section $\theta$ of $E$ is a \emph{contact form}; i.e.\ $\theta\wedge (d\theta)^n$ is nonvanishing.  We will assume further that $(M^{2n+1},J)$ is \emph{strictly pseudoconvex}, meaning that the symmetric tensor $d\theta(\cdot,J\cdot)$ on $H^\ast\otimes H^\ast$ is positive definite; since $E$ is one-dimensional, this is independent of the choice of contact form $\theta$.

Given a CR manifold $(M^{2n+1},J)$, we can define the subbundle $T^{1,0}$ of the complexified tangent bundle $T_{\bC}M$ as the $+i$-eigenspace of $J$, and $T^{0,1}$ as its conjugate.  We likewise denote by $\Lambda^{1,0}$ the space of $(1,0)$-forms --- that is, the subbundle of $T_{\bC}^\ast M$ which annihilates $T^{0,1}$ --- and by $\Lambda^{0,1}$ its conjugate.  The \emph{canonical bundle $K$} is the complex line-bundle $K=\Lambda^{n+1}\left(\Lambda^{1,0}\right)$.

A \emph{pseudohermitian manifold} is a triple $(M^{2n+1},J,\theta)$ of a CR manifold $(M^{2n+1},J)$ together with a choice of contact form $\theta$.  The assumption that $d\theta(\cdot,J\cdot)$ is positive definite implies that the \emph{Levi form $L_\theta(U\wedge\bar V)=-2id\theta(U\wedge\bar V)$} defined on $T^{1,0}$ is a positive-definite Hermitian form.  Since another choice of contact form $\hat\theta$ is equivalent to a choice of (real-valued) function $\sigma\in C^\infty(M)$ such that $\hat\theta=e^\sigma\theta$, and the Levi forms of $\hat\theta$ and $\theta$ are related by $L_{\hat\theta}=e^\sigma L_\theta$, we see that the analogy between CR geometry and conformal geometry begins through the similarity of choosing a contact form or a metric in a conformal class (cf.\ \cite{JerisonLee1987}).

Given a pseudohermitian manifold $(M^{2n+1},J,\theta)$, the \emph{Reeb vector field $T$} is the unique vector field such that $\theta(T)=1$ and $T\contr d\theta=0$.  An \emph{admissible coframe} is a set of $(1,0)$-forms $\{\theta^\alpha\}_{\alpha=1}^n$ whose restriction to $T^{1,0}$ forms a basis for $\left(T^{1,0}\right)^\ast$ and such that $\theta^\alpha(T)=0$ for all $\alpha$.  Denote by $\theta^{\bar\alpha}=\overline{\theta^\alpha}$ the conjugate of $\theta^\alpha$.  Then $d\theta=ih_{\alpha\bar\beta}\theta^\alpha\wedge\theta^{\bar\beta}$ for some positive definite Hermitian matrix $h_{\alpha\bar\beta}$.  Denote by $\{T,Z_\alpha,Z_{\bar\alpha}\}$ the frame for $T_{\bC}M$ dual to $\{\theta,\theta^\alpha,\theta^{\bar\alpha}\}$, so that the Levi form is
\[ L_\theta\left(U^\alpha Z_\alpha,V^{\bar\alpha}Z_{\bar\alpha}\right) = h_{\alpha\bar\beta}U^\alpha V^{\bar\beta} . \]

Tanaka~\cite{Tanaka1975} and Webster~\cite{Webster1977} have defined a canonical connection on a pseudohermitian manifold $(M^{2n+1},J,\theta)$ as follows: Given an admissible coframe $\{\theta^\alpha\}$, define the \emph{connection forms $\omega_\alpha{}^\beta$} and the \emph{torsion form $\tau_\alpha=A_{\alpha\beta}\theta^\beta$} by the relations
\begin{align*}
d\theta^\beta & = \theta^\alpha\wedge\omega_\alpha{}^\beta + \theta\wedge\tau^\beta, \\
\omega_{\alpha\bar\beta}+\omega_{\bar\beta\alpha} & = dh_{\alpha\bar\beta}, \\
A_{\alpha\beta} & = A_{\beta\alpha},
\end{align*}
where we use the metric $h_{\alpha\bar\beta}$ to raise and lower indices; e.g.\ $\omega_{\alpha\bar\beta}=h_{\gamma\bar\beta}\omega_\alpha{}^\gamma$.  In particular, the connection forms are pure imaginary.  The connection forms define \emph{the pseudohermitian connection} on $T^{1,0}$ by $\nabla Z_\alpha=\omega_\alpha{}^\beta\otimes Z_\beta$, which is the unique connection preserving $T^{1,0}$, $T$, and the Levi form.

The curvature form $\Pi_\alpha{}^\beta:=d\omega_\alpha{}^\beta-\omega_\alpha{}^\gamma\wedge\omega_\gamma{}^\beta$ can be written
\[ \Pi_\alpha{}^\beta=R_\alpha{}^\beta{}_{\gamma\bar\delta}\theta^\gamma\wedge\theta^{\bar\delta} \mod \theta , \]
defining the curvature of $M$.  The \emph{pseudohermitian Ricci tensor} is the contraction $R_{\alpha\bar\beta}:=R_\gamma{}^\gamma{}_{\alpha\bar\beta}$ and the \emph{pseudohermitian scalar curvature} is the contraction $R:=R_\alpha{}^\alpha$.  As shown by Webster~\cite{Webster1977}, the contraction $\Pi_\gamma{}^\gamma$ is given by
\begin{equation}
\label{eqn:domegacontraction}
\Pi_\gamma{}^\gamma = d\omega_\gamma{}^\gamma = R_{\alpha\bar\beta}\theta^\alpha\wedge\theta^{\bar\beta} + \nabla^\beta A_{\alpha\beta} \theta^\alpha\wedge\theta - \nabla^{\bar\beta} A_{\bar\alpha\bar\beta} \theta^{\bar\alpha}\wedge\theta .
\end{equation}

For computational and notational efficiency, it will usually be more useful to work with the \emph{pseudohermitian Schouten tensor}
\[ P_{\alpha\bar\beta} := \frac{1}{n+2}\left( R_{\alpha\bar\beta} - \frac{1}{2(n+1)}Rh_{\alpha\bar\beta}\right) \]
and its trace $P:=P_\alpha{}^\alpha=\frac{R}{2(n+1)}$.  The following higher order derivatives
\begin{align*}
T_\alpha & = \frac{1}{n+2}\left(\nabla_\alpha P - i\nabla^\beta A_{\alpha\beta}\right) \\
S & = -\frac{1}{n}\left(\nabla^\alpha T_\alpha + \nabla^{\bar\alpha}T_{\bar\alpha} + P_{\alpha\bar\beta}P^{\alpha\bar\beta} - A_{\alpha\beta}A^{\alpha\beta}\right)
\end{align*}
will also appear frequently (cf.\ \cite{GoverGraham2005,Lee1986}).

In performing computations, we will usually use abstract index notation, so for example $\tau_\alpha$ will denote a $(1,0)$-form and $\nabla_\alpha\nabla_\beta f$ will denote the $(2,0)$-part of the Hessian of a function.  Of course, given an admissible coframe, these expressions give the components of the equivalent tensor.  The following commutator formulae established by Lee~\cite[Lemma~2.3]{Lee1988} will be useful.

\begin{lem}
\label{lem:lee_commutator}
Let $(M^{2n+1},J,\theta)$ be a pseudohermitian manifold.  Then
\begin{align*}
\nabla_\alpha\nabla_\beta f - \nabla_\beta\nabla_\alpha f & = 0, & \nabla_{\bar\beta}\nabla_\alpha f - \nabla_\alpha\nabla_{\bar\beta}f & = ih_{\alpha\bar\beta}\nabla_0f, \\
\nabla_\alpha\nabla_0f - \nabla_0\nabla_\alpha f & = A_{\alpha\gamma}\nabla^\gamma f, & \nabla^\beta\nabla_0\tau_\alpha - \nabla_0\nabla^\beta\tau_\alpha & = A^{\gamma\beta}\nabla_\gamma\tau_\alpha + \tau_\gamma\nabla_\alpha A^{\gamma\beta} ,
\end{align*}
where $\nabla_0$ denotes the derivative in the direction $T$.
\end{lem}

The following consequences of the Bianchi identities established in~\cite[Lemma~2.2]{Lee1988} will also be useful.

\begin{lem}
\label{lem:lee_bianchi}
Let $(M^{2n+1},J,\theta)$ be a pseudohermitian manifold.  Then
\begin{align}
\label{eqn:schouten_bianchi} \nabla^\alpha P_{\alpha\bar\beta} & = \nabla_{\bar\beta}P + (n-1)T_{\bar\beta} \\
\label{eqn:nabla_0r} \nabla_0 R & = \nabla^\alpha\nabla^\beta A_{\alpha\beta} + \nabla_\alpha\nabla_\beta A^{\alpha\beta} .
\end{align}
\end{lem}

In particular, combining the results of Lemma~\ref{lem:lee_commutator} and Lemma~\ref{lem:lee_bianchi} yields
\begin{equation}
\label{eqn:subplacian_R}
\Delta_b R - 2n\Imaginary\nabla^\alpha\nabla^\beta A_{\alpha\beta} = -2\nabla^\alpha\left(\nabla_\alpha R - in\nabla^\beta A_{\alpha\beta}\right) .
\end{equation}

An important operator in the study of pseudohermitian manifolds is the \emph{sublaplacian}
\[ \Delta_b := -\left(\nabla^\alpha\nabla_\alpha + \nabla_\alpha\nabla^\alpha\right) . \]
Defining the \emph{subgradient $\nabla_bu$} as the projection of $du$ onto $H^\ast\otimes\bC$ --- that is, $\nabla_bf=\nabla_\alpha f+\nabla_{\bar\alpha}f$ --- it is easy to show that
\[ \int_M u\Delta_b v\,\theta\wedge d\theta^n = \int_M \lp\nabla_b u,\nabla_b v\rp \theta\wedge d\theta^n \]
for any $u,v\in C^\infty(M)$, at least one of which is compactly supported, and where $\lp\cdot,\cdot\rp$ denotes the Levi form.

One important consequence of Lemma~\ref{lem:lee_commutator} is that the operator $C$ has the following two equivalent forms:
\begin{equation}
\label{eqn:subplacian_squared}
\begin{split}
Cf & := \Delta_b^2f + n^2\nabla_0^2f - 2in\nabla_\beta\left(A^{\alpha\beta}\nabla_\alpha f\right) + 2in\nabla^\beta\left(A_{\alpha\beta}\nabla^\alpha f\right) \\
& = 4\nabla^\alpha\left(\nabla_\alpha\nabla_\beta\nabla^\beta f + in A_{\alpha\beta}\nabla^\beta f\right) .
\end{split}
\end{equation}
In dimension $n=1$, the operator $C$ is the compatibility operator found by Graham and Lee~\cite{GrahamLee1988}.  Hirachi~\cite{Hirachi1990} later observed that in this dimension $C$ is a CR covariant operator, in the sense that it satisfies a particularly simple transformation formula under a change of contact form.  Thus, in this dimension $C$ is the CR Paneitz operator $P_4$; for further discussion, see Section~\ref{sec:defn}.

\subsection{CR pluriharmonic functions}
\label{sec:bg/pluriharmonic}

Given a CR manifold $(M^{2n+1},J)$, a \emph{CR pluriharmonic function} is a function $u\in C^\infty(M)$ which is locally the real part of a \emph{CR function} $v\in C^\infty(M;\bC)$; i.e.\ $u=\Real(v)$ for $v$ satisfying $\dbar v:=\nabla_{\bar\alpha}v=0$.  We will denote by $\mP$ the space of pluriharmonic functions on $M$, which is usually an infinite-dimensional vector space.  When additionally a choice of contact form $\theta$ is given, Lee~\cite{Lee1988} proved the following alternative characterization of CR pluriharmonic functions which does not require solving $\dbar v=0$.

\begin{prop}
\label{prop:pluriharmonic}
Let $(M^{2n+1},J,\theta)$ be a pseudohermitian manifold.  A function $u\in C^\infty(M)$ is CR pluriharmonic if and only if
\begin{align*}
B_{\alpha\bar\beta}u := \nabla_{\bar\beta}\nabla_\alpha u - \frac{1}{n}\nabla^\gamma\nabla_\gamma u\,h_{\alpha\bar\beta} & = 0, \qquad \text{if $n\geq 2$} \\
P_\alpha u := \nabla_\alpha\nabla_\beta\nabla^\beta u + inA_{\alpha\beta}\nabla^\beta u & = 0, \qquad \text{if $n=1$} .
\end{align*}
\end{prop}

Using Lemma~\ref{lem:lee_commutator}, it is straightforward to check that (cf.\ \cite{GrahamLee1988})
\begin{equation}
\label{eqn:grahamlee}
\nabla^{\bar\beta}\left(B_{\alpha\bar\beta}u\right) = \frac{n-1}{n}P_\alpha u .
\end{equation}
In particular, we see that the vanishing of $B_{\alpha\bar\beta}u$ implies the vanishing of $P_\alpha u$ when $n>1$.  Moreover, the condition $B_{\alpha\bar\beta}u=0$ is vacuous when $n=1$, and by~\eqref{eqn:grahamlee}, we can consider the condition $P_\alpha u=0$ from Proposition~\ref{prop:pluriharmonic} as the ``residue'' of the condition $B_{\alpha\bar\beta}u=0$ (cf.\ Section~\ref{sec:lee} and Section~\ref{sec:defn}).

Note also that, using the second expression in~\eqref{eqn:subplacian_squared}, we have that $C=4\nabla^\alpha P_\alpha$.  In particular, it follows that $\mP\subset\ker P_4$ for three-dimensional CR manifolds $(M^3,J)$.  It is easy to see that this is an equality when $(M^3,J)$ admits a torsion-free contact form (cf.\ \cite{GrahamLee1988}), but a good characterization of when equality holds is not yet known.

\subsection{CR density bundles}
\label{sec:bg/density}

One generally wants to study CR geometry using \emph{CR invariants}; i.e.\ using invariants of a CR manifold $(M^{2n+1},J)$.  However, it is frequently easier to do geometry by making a choice of contact form $\theta$ so as to make use of the Levi form and the associated pseudohermitian connection.  If one takes this point of view, it then becomes important to know how the pseudohermitian connection and the pseudohermitian curvatures transform under a change of contact form, and also to have a convenient way to describe objects which transform in a simple way with a change of contact form.  This goal is met using CR density bundles.

Given a CR manifold $(M^{2n+1},J)$, choose a $(n+2)$-nd root of the canonical bundle $K$ and denote it by $E(1,0)$; this can always be done locally, and since we are entirely concerned with local invariants, this poses no problems.  Given any $w,w^\prime\in\bR$ with $w-w^\prime\in\bZ$, the \emph{$(w,w^\prime)$-density bundle $E(w,w^\prime)$} is the complex line bundle
\[ E(w,w^\prime) = E(1,0)^{\otimes w} \otimes \overline{E(1,0)}^{\otimes w^\prime} . \]
For our purposes, the important property of $E(w,w^\prime)$ is that a choice of contact form $\theta$ induces an isomorphism between the space $\mE(w,w^\prime)$ of smooth sections of $E(w,w^\prime)$ and $C^\infty(M;\bC)$,
\[ \mE(w,w^\prime) \ni u \cong u_\theta \in C^\infty(M;\bC), \]
with the property that if $\hat\theta=e^\sigma\theta$ is another choice of contact form, then $u_{\hat\theta}$ is related to $u_\theta$ by
\begin{equation}
\label{eqn:complex_transformation}
u_{\hat\theta} = e^{\frac{w}{2}\sigma}e^{\frac{w^\prime}{2}\bar\sigma} u_\theta ;
\end{equation}
for details, see~\cite{GoverGraham2005}.  We will also consider density-valued tensor bundles; for example, we will denote by $E^\alpha(w,w^\prime)$ and $E^{\bar\alpha}(w,w^\prime)$ the tensor products $T^{1,0}\otimes E(w,w^\prime)$ and $T^{0,1}\otimes E(w,w^\prime)$, respectively, and by $\mE^\alpha(w,w^\prime)$ and $\mE^{\bar\alpha}(w,w^\prime)$ their respective spaces of smooth sections.  In this way, we may regard the Levi form as the density $h_{\alpha\bar\beta}\in\mE_{\alpha\bar\beta}(1,1)$, thereby suppressing the exponential factor which normally appears when writing how it transforms under a change of contact structure.

Since we will be primarily interested in real-valued functions and tensors, our primary interest will be in the $(w,w)$-density bundles and tensor products thereof.  In particular, if $u\in\mE(w,w)$ is real-valued and we restrict ourselves to real-valued contact forms, the transformation rule~\eqref{eqn:complex_transformation} becomes $u_{\hat\theta}=e^{w\sigma}u_\theta$.

In~\cite{Lee1988} (see also~\cite{GoverGraham2005}), the transformation formulae for the pseudohermitian connection and its torsion and curvatures under a change of contact form are given, which we record below:

\begin{lem}
\label{lem:cr_change}
Let $(M^{2n+1},J,\theta)$ be a pseudohermitian manifold and regard the torsion $A_{\alpha\beta}\in\mE_{\alpha\beta}(0,0)$, the pseudohermitian Schouten tensor $P_{\alpha\bar\beta}\in\mE_{\alpha\bar\beta}(0,0)$, and its trace $P=P_\alpha{}^\alpha\in\mE(-1,-1)$.  Additionally, let $f\in\mE(w,w)$ and $\tau_\alpha\in\mE_\alpha(w,w)$.  If $\hat\theta=e^\sigma\theta$ is another choice of contact form and $\hat A_{\alpha\beta},\hat P_{\alpha\bar\beta},\hat P$ are its torsion, pseudohermitian Schouten tensor, and its trace, respectively, then
\begin{align*}
\hat A_{\alpha\beta} & = A_{\alpha\beta} + i\nabla_\beta\nabla_\alpha\sigma - i(\nabla_\alpha\sigma)(\nabla_\beta\sigma) \\
\hat P_{\alpha\bar\beta} & = P_{\alpha\bar\beta} - \frac{1}{2}\left(\nabla_{\bar\beta}\nabla_\alpha\sigma+\nabla_\alpha\nabla_{\bar\beta}\sigma\right) - \frac{1}{2}\lv\nabla_\gamma\sigma\rv^2 h_{\alpha\bar\beta} \\
\hat P & = P + \frac{1}{2}\Delta_b\sigma - \frac{n}{2}\lv\nabla_\gamma\sigma\rv^2 \\
\hat\nabla_\alpha f & = \nabla_\alpha f + wf\nabla_\alpha\sigma \\
\hat\nabla_0 f & = \nabla_0f + i(\nabla_\alpha\sigma)(\nabla^\alpha f) - i(\nabla^\alpha\sigma)(\nabla_\alpha f) + wf\nabla_0\sigma \\
\hat\nabla_\alpha\tau_\beta & = \nabla_\alpha\tau_\beta + (w-1)\tau_\beta\nabla_\alpha\sigma - \tau_\alpha\nabla_\beta\sigma \\
\hat\nabla_{\bar\beta}\tau_\alpha & = \nabla_{\bar\beta}\tau_\alpha + w\tau_\alpha\nabla_{\bar\beta}\sigma + \tau_\gamma\nabla^\gamma\sigma h_{\alpha\bar\beta} .
\end{align*}
\end{lem}

There are a few technical comments necessary to properly interpret Lemma~\ref{lem:cr_change}.

First, we define the norm $\lv\nabla_\gamma\sigma\rv^2:=(\nabla_\gamma\sigma)(\nabla^\gamma\sigma)$.  In particular, $\lv\nabla_\gamma\sigma\rv^2=\frac{1}{2}\lp\nabla_b\sigma,\nabla_b\sigma\rp$.  We define norms on all (density-valued) tensors in a similar way; for example, $\lv A_{\alpha\beta}\rv^2=A_{\alpha\beta}A^{\alpha\beta}$ and $\lv P_{\alpha\bar\beta}\rv^2=P_{\alpha\bar\beta}P^{\alpha\bar\beta}$.

Second, for these formulae to be valid component-wise, one also needs to change the admissible frame in which one computes the components of the torsion and CR Schouten tensor.  Explicitly, if $\{\theta,\theta^\alpha,\theta^{\bar\alpha}\}$ is an admissible coframe for the contact form $\theta$, one defines
\[ \hat\theta^\alpha = \theta^\alpha + i(\nabla^\alpha\sigma)\theta \]
and $\hat\theta^{\bar\alpha}$ by conjugation, ensuring that $\{\hat\theta,\hat\theta^\alpha,\hat\theta^{\bar\alpha}\}$ is an admissible coframe for the contact form $\hat\theta$.  In the above formulae, this frame is used to compute the components of $\hat\nabla_\alpha$ and $\hat\nabla_{\bar\alpha}$, while the coframe $\{\theta,\theta^\alpha,\theta^{\bar\alpha}\}$ is used to compute the components of $\nabla_\alpha$ and $\nabla_{\bar\alpha}$.

Third, to regard $P\in\mE(-1,-1)$ means to extend the \emph{function} $P$ to a \emph{density} $\rho\in\mE(-1,-1)$ by requiring $\rho_\theta=P$, and we use $h_{\alpha\bar\beta}\in\mE_{\alpha\bar\beta}(1,1)$ to raise and lower indices.  This has the effect that, at the level of functions, Lemma~\ref{lem:cr_change} states that
\[ \hat P = e^{-\sigma}\left( P + \frac{1}{2}\Delta_b\sigma - \frac{n}{2}\lv\nabla_\gamma\sigma\rv^2\right) , \]
which is the transformation formula proven in~\cite{Lee1988}.  It also means that we can quickly compute how $\nabla_\alpha P$ transforms under a change of contact form: Using Lemma~\ref{lem:cr_change} with $P\in\mE(-1,-1)$, it follows immediately that
\[ \widehat{\nabla_\alpha} \widehat{P} = \nabla_\alpha\left(P+\frac{1}{2}\Delta_b\sigma-\frac{n}{2}\lv\nabla_\gamma\sigma\rv^2\right) - \left(P+\frac{1}{2}\Delta_b\sigma-\frac{n}{2}\lv\nabla_\gamma\sigma\rv^2\right)\nabla_\alpha\sigma . \]
These conventions will be exploited heavily in Section~\ref{sec:covariance}.

%% file: lee.tex
\section{Pseudo-Einstein contact forms in three dimensions}
\label{sec:lee}

In~\cite{Lee1988}, Lee defined pseudo-Einstein manifolds as pseudohermitian manifolds $(M^{2n+1},J,\theta)$ such that $P_{\alpha\bar\beta}-\frac{1}{n}Ph_{\alpha\bar\beta}=0$ and studied their existence when $n\geq 2$.  In particular, he showed in this case that $\theta$ is pseudo-Einstein if and only if it is locally volume-normalized with respect to a closed nonvanishing section of $K$; that is, using the terminology of Fefferman and Hirachi~\cite{FeffermanHirachi2003}, $\theta$ is pseudo-Einstein if and only if it is an invariant contact form.  While Lee's definition of pseudo-Einstein contact forms is vacuous in dimension three, the notion of an invariant contact form is not.  It turns out that, analogous to Proposition~\ref{prop:pluriharmonic}, there is a meaningful way to extend the notion of pseudo-Einstein contact forms to the case $n=1$ as a higher order condition on $\theta$ which retains the equivalence with invariant contact forms.  As this notion will be essential to our discussion of the $Q^\prime$-curvature, and because it did not appear elsewhere in the literature at the time the work of this paper was being completed, we devote this section to explaining this three-dimensional notion of pseudo-Einstein contact forms.

\begin{defn}
\label{defn:pseudo_einstein}
A pseudohermitian manifold $(M^{2n+1},J,\theta)$ is said to be \emph{pseudo-Einstein} if
\begin{align*}
R_{\alpha\bar\beta} - \frac{1}{n}Rh_{\alpha\bar\beta} & = 0, \qquad \text{if $n\geq 2$}, \\
\nabla_\alpha R - i\nabla^\beta A_{\alpha\beta} & = 0, \qquad \text{if $n=1$}.
\end{align*}
\end{defn}

One way to regard this definition is as an analogue of Lee's characterization~\cite{Lee1988} of CR pluriharmonic functions from Proposition~\ref{prop:pluriharmonic}.  Indeed, a straightforward computation using Lemma~\ref{lem:lee_bianchi} shows that
\begin{equation}
\label{eqn:divtrfreep}
\nabla^{\bar\beta}\left( R_{\alpha\bar\beta} - \frac{1}{n}Rh_{\alpha\bar\beta}\right) = \frac{n-1}{n}\left(\nabla_\alpha R - in\nabla^\beta A_{\alpha\beta}\right)
\end{equation}
holds for any CR manifold $(M^{2n+1},J,\theta)$.  In particular, our definition of a three-dimensional pseudo-Einstein manifold can be regarded as the ``residue'' of the usual definition when $n\geq2$.

The characterization of pseudo-Einstein contact forms as invariant contact forms persists in the case $n=1$.  To see this, let us first recall what it means for a contact form to be volume-normalized with respect to a section of $K$.

\begin{defn}
Given a CR manifold $(M^{2n+1},J)$ and a nonvanishing section $\omega$ of the canonical bundle $K$, we say that a contact form $\theta$ is \emph{volume-normalized with respect to $\omega$} if
\[ \theta\wedge(d\theta)^n = i^{n^2}n! \theta\wedge(T\contr\omega)\wedge(T\contr\overline{\omega}) . \]
\end{defn}

By considering all the terms in $d\omega_\alpha{}^\alpha$, Lee's argument~\cite[Theorem~4.2]{Lee1988} establishing the equivalence between pseudo-Einstein contact forms and invariant contact forms can be extended to the case $n=1$.

\begin{thm}
\label{thm:lee}
Let $(M^3,J)$ be a three-dimensional CR manifold.  A contact form $\theta$ on $M$ is pseudo-Einstein if and only if for each point $p\in M$, there exists a neighborhood of $p$ in which there is a closed section of the canonical bundle with respect to which $\theta$ is volume-normalized.
\end{thm}

The main step in the proof of Theorem~\ref{thm:lee} is the following analogue of~\cite[Lemma~4.1]{Lee1988}.

\begin{lem}
\label{lem:lee}
Let $(M^3,J)$ be a three-dimensional CR manifold.  A contact form $\theta$ on $M$ is pseudo-Einstein if and only if with respect to any admissible coframe $\{\theta,\theta^\alpha,\theta^{\bar\alpha}\}$ the one-form $\omega_\alpha{}^\alpha+iR\theta$ is closed.
\end{lem}

\begin{proof}

Using~\eqref{eqn:domegacontraction} and the assumption $n=1$, it holds in general that
\begin{equation}
\label{eqn:domega}
d\omega_\alpha{}^\alpha = Rh_{\alpha\bar\beta}\theta^\alpha\wedge\theta^{\bar\beta} + \nabla^\beta A_{\alpha\beta} \theta^\alpha\wedge\theta - \nabla^{\bar\beta}A_{\bar\alpha\bar\beta}\theta^{\bar\alpha}\wedge\theta .
\end{equation}
It thus follows that
\[ d\left(\omega_\alpha{}^\alpha+iR\theta\right) = 2i\Real\left((\nabla_\alpha R-i\nabla^\beta A_{\alpha\beta})\theta^\alpha\wedge\theta\right), \]
from which the conclusion follows immediately.
\end{proof}

\begin{proof}[Proof of Theorem~\ref{thm:lee}]

First suppose that $\theta$ is volume-normalized with respect to a closed section $\xi\in K$ on a neighborhood $U$ of $p$, and choose an admissible coframe $\{\theta,\theta^\alpha,\theta^{\bar\alpha}\}$ such that $d\theta=i\theta^\alpha\wedge\theta^{\bar\alpha}$.  Since $\xi\in K$, there is a function $\lambda\in C^\infty(M,\bC)$ such that $\xi=\lambda\theta\wedge\theta^\alpha$.  On the other hand, since $\theta$ is volume-normalized with respect to $\xi$, it must hold that $\lv\lambda\rv=1$.  Thus, upon replacing $\theta^\alpha$ by $\lambda^{-1}\theta^\alpha$, we have that $\xi=\theta\wedge\theta^\alpha$.

Now, using the definition of the connection one-form $\omega_\alpha{}^\alpha$, it holds in general that
\begin{equation}
\label{eqn:dxi}
d\xi = -\omega_\alpha{}^\alpha\wedge\xi .
\end{equation}
Since $\xi$ is closed, this shows that $\omega_\alpha{}^\alpha$ is a $(1,0)$-form.  But $\omega_\alpha{}^\alpha$ is also pure imaginary, hence $\omega_\alpha{}^\alpha=iu\theta$ for some $u\in C^\infty(M)$.  Differentiating, we see that
\[ d\omega_\alpha{}^\alpha = -u\theta^\alpha\wedge\theta^{\bar\alpha} + i\nabla_\alpha u\theta^\alpha\wedge\theta + i\nabla_{\bar\alpha}u\theta^{\bar\alpha}\wedge\theta . \]
It thus follows from~\eqref{eqn:domega} that $R=-u$ and $\nabla^\beta A_{\alpha\beta}=i\nabla_\alpha u$.  In particular, $\theta$ is pseudo-Einstein.

Conversely, suppose that $\theta$ is pseudo-Einstein.  In a neighborhood of $p\in M$, let $\{\theta,\theta^\alpha,\theta^{\bar\alpha}\}$ be an admissible coframe such that $d\theta=i\theta^\alpha\wedge\theta^{\bar\alpha}$, and define $\xi_0=\theta\wedge\theta^\alpha\in K$.  By~\eqref{eqn:dxi} it holds that $d\xi_0=-\omega_\alpha{}^\alpha\wedge\xi_0$, while by Lemma~\ref{lem:lee} there exists a function $\phi$ such that
\[ \omega_\alpha{}^\alpha + iR\theta = id\phi . \]
Since $\omega_\alpha{}^\alpha$ is pure imaginary, we can take $\phi$ to be real, whence $d\left(e^{i\phi}\xi_0\right)=0$.  Since $\theta$ is volume-normalized with respect to $e^{i\phi}\xi_0$, this gives the desired section of $K$.
\end{proof}

Another nice property of pseudo-Einstein contact forms when $n\geq2$ is that, when they exist, they can be characterized in terms of CR pluriharmonic functions.  This too persists in the case $n=1$, which is crucial to making sense of the $Q^\prime$-curvature.  To see this, let $(M^3,J,\theta)$ be a pseudohermitian manifold and define the $(1,0)$-form $W_\alpha$ by
\[ W_\alpha := \nabla_\alpha R - i\nabla^\beta A_{\alpha\beta} . \]
Observe that $W_\alpha$ vanishes if and only if $\theta$ is pseudo-Einstein.  As first observed by Hirachi~\cite{Hirachi1990}, $W_\alpha$ satisfies a simple transformation formula; given another contact form $\hat\theta=e^\sigma\theta$, a straightforward computation using Lemma~\ref{lem:cr_change} shows that
\begin{equation}
\label{eqn:Walpha}
\hat W_\alpha = W_\alpha - 3P_\alpha\sigma,
\end{equation}
where here we regard $W_\alpha\in\mE_\alpha(-1,-1)$.  An immediate consequence of~\eqref{eqn:Walpha} is the following correspondence between pseudo-Einstein contact forms and CR pluriharmonic functions.

\begin{prop}
\label{prop:pluriharmonic_psie}
Let $(M^3,J,\theta)$ be a pseudo-Einstein three-manifold.  Then the set of pseudo-Einstein contact forms on $(M^3,J)$ is given by
\[ \left\{ e^u\theta \colon u \text{ is a CR pluriharmonic function} \right\} . \]
\end{prop}

Following~\cite{Lee1988}, there are topological obstructions to the existence of an invariant contact form $\theta$ on a three-dimensional CR manifold $(M^3,J)$.  However, if $(M^3,J)$ is the boundary of a strictly pseudoconvex domain in $\bC^2$, then there always exists a closed section of $K$, and hence a pseudo-Einstein contact form.  This is a slight refinement of the observation by Fefferman and Hirachi~\cite{FeffermanHirachi2003} that any such CR manifold admits a contact form $\theta$ such that the $Q$-curvature
\begin{equation}
\label{eqn:Q_divergence}
Q := -\frac{4}{3}\nabla^\alpha W_\alpha
\end{equation}
vanishes.

%% file: defn.tex
\section{The CR Paneitz and $Q$-Curvature Operators}
\label{sec:defn}

The CR Paneitz operator in dimension three is well-known and given by $P_4=C$.  However, in higher dimensions the operator $C$ is not CR covariant.  The correct definition, in that $P_4$ is CR covariant, is as follows.

\begin{defn}
\label{defn:cr_paneitz}
Let $(M^{2n+1},J,\theta)$ be a CR manifold.  The \emph{CR Paneitz operator $P_4$} is the operator
\begin{align*}
P_4f & := \Delta_b^2 f + \nabla_0^2 f - 4\Imaginary\left(\nabla^\alpha(A_{\alpha\beta}\nabla^\beta f)\right) + 4\Real\left(\nabla_{\bar\beta}(\tracefree{P}{}^{\alpha\bar\beta}\nabla_\alpha f)\right) \\
& \quad - \frac{4(n^2-1)}{n}\Real\left(\nabla^\beta(P\nabla_\beta f)\right) + \frac{n-1}{2}Q_4f .
\end{align*}
where
\begin{align*}
Q_4 & = \frac{2(n+1)^2}{n(n+2)}\Delta_b P - \frac{4}{n(n+2)}\Imaginary\left(\nabla^\alpha\nabla^\beta A_{\alpha\beta}\right) - \frac{2(n-1)}{n}\lv A_{\alpha\beta}\rv^2 \\
& \quad - \frac{2(n+1)}{n}\lv \tracefree{P_{\alpha\bar\beta}}\rv^2 + \frac{2(n-1)(n+1)^2}{n^2}P^2
\end{align*}
and $\tracefree{P_{\alpha\bar\beta}}=P_{\alpha\bar\beta} - \frac{P}{n}h_{\alpha\bar\beta}$ is the tracefree part of the CR Schouten tensor.
\end{defn}

The above expression for the CR Paneitz operator in general dimension does not seem to appear anywhere in the literature, though its existence and two different methods to derive the formula have been established by Gover and Graham~\cite{GoverGraham2005}.  In particular, their construction immediately implies that the CR Paneitz operator is CR covariant,
\begin{equation}
\label{eqn:cr_covariant}
P_4 \colon \mE\left(-\frac{n-1}{2},-\frac{n-1}{2}\right) \to \mE\left(-\frac{n+3}{2},-\frac{n+3}{2}\right) .
\end{equation}
By inspection, it is clear that $P_4$ is a real, (formally) self-adjoint fourth order operator of the form $\Delta_b^2+T^2$ plus lower order terms, and thus has the form one expects of a ``Paneitz operator'' (cf.\ \cite{GoverGraham2005}).  For convenience, we derive in the appendices the above expression for the CR Paneitz operator using both methods described in~\cite{GoverGraham2005}, namely the CR tractor calculus and restriction from the Fefferman bundle.

As mentioned in Section~\ref{sec:bg}, in the critical case $n=1$ we have that $\mP\subset\ker P_4$.  Motivated by~\cite{BransonFontanaMorpurgo2007,BransonGover2005}, we define the $P^\prime$-operator corresponding to the CR Paneitz operator as a renormalization of the part of $P_4$ which doesn't annihilate pluriharmonic functions.

\begin{defn}
\label{defn:q}
Let $(M^{2n+1},J,\theta)$ be a CR manifold.  The \emph{$P^\prime$-operator $P_4^\prime\colon\mP\to C^\infty(M)$} is defined by
\[ P_4^\prime f = \frac{2}{n-1}P_4f . \]
When $n=1$, we define $P_4^\prime$ by the formal limit
\[ P_4^\prime f = \lim_{n\to1} \frac{2}{n-1}P_4f . \]
\end{defn}

The key property of the $P^\prime$-operator, which we check explicitly below, is that the expression for $P_4^\prime$ as defined in Definition~\ref{defn:q} is rational in the dimension and does not have a pole at $n=1$; in particular, it is meaningful to discuss the $P^\prime$-operator on three-dimensional CR manifolds.

\begin{lem}
\label{lem:q}
Let $(M^{2n+1},J,\theta)$ be a CR manifold.  Then the $P^\prime$-operator is given by
\begin{equation}
\label{eqn:q}
\begin{split}
P_4^\prime f & = \frac{2(n+1)}{n^2}\Delta_b^2f - \frac{8}{n}\Imaginary\left(\nabla^\alpha(A_{\alpha\beta}\nabla^\beta f)\right) - \frac{8(n+1)}{n}\Real\left(\nabla^\alpha(P\nabla_\alpha f)\right) \\
& \quad + \frac{16(n+1)}{n(n+2)}\Real\left((\nabla_\alpha P-\frac{in}{2(n+1)}\nabla^\beta A_{\alpha\beta})\nabla^\alpha f\right) \\
& \quad + \bigg[\frac{2(n+1)^2}{n(n+2)}\Delta_b P - \frac{4}{n(n+2)}\Imaginary\left(\nabla^\alpha\nabla^\beta A_{\alpha\beta}\right) \\
& \qquad - \frac{2(n-1)}{n}\lv A_{\alpha\beta}\rv^2 - \frac{2(n+1)}{n}\lv\tracefree{P_{\alpha\bar\beta}}\rv^2 + \frac{2(n-1)(n+1)^2}{n^2}P^2\bigg] f .
\end{split}
\end{equation}
In particular, if $n=1$, the \emph{critical $P^\prime$-operator} is given by
\begin{equation}
\label{eqn:q_crit}
\begin{split}
P_4^\prime f & = 4\Delta_b^2 f - 8\Imaginary\left(\nabla^\alpha(A_{\alpha\beta}\nabla^\beta f)\right) - 4\Real\left(\nabla^\alpha(R\nabla_\alpha f)\right) \\
& \quad + \frac{8}{3}\Real\left((\nabla_\alpha R - i\nabla^\beta A_{\alpha\beta})\nabla^\alpha f\right) + \frac{2}{3}\left(\Delta_b R - \frac{1}{2}\Imaginary\nabla^\alpha\nabla^\beta A_{\alpha\beta}\right) f .
\end{split}
\end{equation}
\end{lem}

\begin{proof}

When $n>1$, this follows directly from the definition of the CR Paneitz operator and the fact that $f\in\mP$ if and only if
\[ \nabla_{\bar\beta}\nabla_\alpha f = \mu\,h_{\alpha\bar\beta} \]
for some $\mu\in C^\infty(M)$, which in turn implies, using~\eqref{eqn:subplacian_squared}, that
\[ \Delta_b^2 f + n^2\nabla_0^2 f - 4n\Imaginary\left(\nabla^\alpha(A_{\alpha\beta}\nabla^\beta f)\right) = 0 . \]
Letting $n\to1$ then yields the case $n=1$.
\end{proof}

Note that, as an operator on $C^\infty(M)$, the $P^\prime$-operator is only determined uniquely up to the addition of operators which annihilate $\mP$.  We have chosen the expression~\eqref{eqn:q_crit} so that our expression does not involve $T$-derivatives.  In particular, this allows us to readily connect the $P^\prime$-operator to similar objects already appearing in the literature.
\begin{enumerate}
\item On a general CR manifold $(M^3,J,\theta)$,
\[ P_4^\prime(1) = \frac{2}{3}\left(\Delta_b R - 2\Imaginary\nabla^\alpha\nabla^\beta A_{\alpha\beta}\right) , \]
which is, using~\eqref{eqn:subplacian_R}, Hirachi's $Q$-curvature~\eqref{eqn:Q_divergence}.
\item On $(S^3,J,\theta)$ with its standard CR structure, the $P^\prime$-operator is given by
\[ P_4^\prime = 4\Delta_b^2 + 2\Delta_b, \]
which is the operator introduced by Branson, Fontana and Morpurgo~\cite{BransonFontanaMorpurgo2007}.
\end{enumerate}

The means by which we defined the $P^\prime$-operator, and which we will further employ to establish its CR covariance, is called ``analytic continuation in the dimension'' (cf.\ \cite{Branson1995}).  However, due to the relatively simple form of the expression for $P_4^\prime$, we can also check its CR covariance directly, as is carried out in Section~\ref{sec:covariance}.

In the case that $(M^3,J,\theta)$ is a pseudo-Einstein manifold, the $P^\prime$-operator takes the simple form
\begin{equation}
\label{eqn:p4prime_invariant_contact_form}
P_4^\prime = 4\Delta_b^2 -8\Imaginary\nabla^\alpha\left(A_{\alpha\beta}\nabla^\beta\right) - 4\Real\nabla^\alpha\left(R\nabla_\alpha\right) .
\end{equation}
In particular, we see that $P_4^\prime$ annihilates constants, leading us to consider the ``$Q$-curvature of the $P^\prime$-operator,'' which we shall simply call the $Q^\prime$-curvature.

\begin{defn}
Let $(M^{2n+1},J,\theta)$ be a pseudo-Einstein manifold.  The \emph{$Q^\prime$-curvature $Q_4^\prime\in C^\infty(M)$} is the local invariant defined by
\[ Q_4^\prime = \frac{2}{n-1}P_4^\prime(1) = \frac{4}{(n-1)^2}P_4(1) . \]
When $n=1$, we define $Q_4^\prime$ as the formal limit
\[ Q_4^\prime = \lim_{n\to 1}\frac{4}{(n-1)^2}P_4(1) . \]
\end{defn}

Again, it is straightforward to give an explicit formula for $Q_4^\prime$.

\begin{lem}
\label{lem:qprime}
Let $(M^{2n+1},J,\theta)$ be a pseudo-Einstein manifold.  Then the $Q^\prime$-curvature is given by
\begin{equation}
\label{eqn:q4prime}
Q_4^\prime = \frac{2}{n^2}\Delta_b R - \frac{4}{n}\lv A_{\alpha\beta}\rv^2 + \frac{1}{n^2}R^2 .
\end{equation}
In particular, when $n=1$ the $Q^\prime$-curvature is
\begin{equation}
\label{eqn:q4prime_crit}
Q_4^\prime = 2\Delta_b R - 4\lv A_{\alpha\beta}\rv^2 + R^2 .
\end{equation}
\end{lem}

\begin{proof}

When $n>1$, it follows from~\eqref{eqn:divtrfreep} and the pseudo-Einstein assumption that $\lv\tracefree{P_{\alpha\bar\beta}}\rv^2=0$ and
\[ \Delta_b R - 2n\Imaginary\left(\nabla^\alpha\nabla^\beta A_{\alpha\beta}\right) = 0 . \]
Plugging in to~\eqref{eqn:q}, we see that
\[ P_4^\prime(1) = \frac{n-1}{n^2}\Delta_b R - \frac{2(n-1)}{n}\lv A_{\alpha\beta}\rv^2 + \frac{(n-1)}{2n^2}R^2 . \]
Multiplying by $\frac{2}{n-1}$ then yields the desired result.
\end{proof}

Let us now verify some basic properties of the $P^\prime$-operator and the $Q^\prime$-curvature.  These objects are best behaved and the most interesting in the critical dimension $n=1$, so we shall make our statements only in this dimension.

\begin{prop}
\label{prop:q_covariant}
Let $(M^3,J,\theta)$ be a pseudohermitian manifold with $P^\prime$-operator $P_4^\prime\colon\mP\to C^\infty(M)$.  Then the following properties hold.
\begin{enumerate}
\item $P_4^\prime$ is formally self-adjoint.
\item Given another choice of contact form $\hat\theta=e^\sigma\theta$ with $\sigma\in C^\infty(M)$, it holds that
\begin{equation}
\label{eqn:q_operator_covariant}
e^{2\sigma}\hat P_4^\prime(f) = P_4^\prime(f) + P_4(f\sigma)
\end{equation}
for all $f\in\mP$, where $\hat P_4^\prime$ denotes the $P_4^\prime$-operator defined in terms of $\hat\theta$.
\end{enumerate}
\end{prop}

\begin{proof}

On a general pseudohermitian manifold $(M^{2n+1},J,\theta)$, it follows from Definition~\ref{defn:q} and the self-adjointness of $P_4$ that, given $u,v\in\mP$,
\[ \frac{n-1}{2}\int_M u\,P_4^\prime v = \int_M u\,P_4v = \int_M v\,P_4u = \frac{n-1}{2}\int_M v\,P_4^\prime u, \]
establishing the self-adjointness of $P_4^\prime$.  Likewise, the covariance~\eqref{eqn:cr_covariant} of the CR Paneitz operator implies that for all $u\in\mP$,
\[ \frac{n-1}{2}e^{\frac{n+3}{2}\sigma}\hat P_4^\prime(u) = P_4\left(e^{\frac{n-1}{2}\sigma}u\right) = \frac{n-1}{2}P_4^\prime(u) + P_4\left(\left(e^{\frac{n-1}{2}\sigma}-1\right)u\right) . \]
Multiplying both sides by $\frac{2}{n-1}$ and taking the limit $n\to 1$ yields~\eqref{eqn:q_operator_covariant}.
\end{proof}

\begin{remark}

It would be nice to have a formula for the critical $P_4^\prime$-operator which is manifestly formally self-adjoint on \emph{all} functions.  At present, we have not been able to find such a formula without the assumption that $\theta$ is a pseudo-Einstein contact form, in which case~\eqref{eqn:p4prime_invariant_contact_form} gives such a formula.
\end{remark}

Using the same argument with $Q_4^\prime$ in place of $P_4^\prime$ and $P_4^\prime$ in place of $P_4$, we get a similar result for transformation law of the $Q^\prime$-curvature when the contact form is changed by a CR pluriharmonic function.

\begin{prop}
\label{prop:qprime_covariant}
Let $(M^3,J,\theta)$ be a pseudo-Einstein manifold.  Given $\sigma\in\mP$, denote $\hat\theta=e^\sigma\theta$.  Then
\begin{equation}
\label{eqn:qprime_operator_covariant}
e^{2\sigma}\hat Q_4^\prime = Q_4^\prime + P_4^\prime(\sigma) + \frac{1}{2}P_4(\sigma^2) .
\end{equation}
In particular, if $M$ is compact then
\begin{equation}
\label{eqn:qprime_integral}
\int_M \hat Q_4^\prime\,\hat\theta\wedge d\hat\theta = \int_M Q_4^\prime\,\theta\wedge d\theta .
\end{equation}
\end{prop}

\begin{proof}

For $n>1$ and $\sigma\in\mP$, we have that
\begin{align*}
\left(\frac{n-1}{2}\right)^2 e^{\frac{n+3}{2}\sigma}\hat Q_4^\prime & = \left(\frac{n-1}{2}\right)^2 Q_4^\prime + P_4\left(e^{\frac{n-1}{2}\sigma} - 1 \right) \\
& = \left(\frac{n-1}{2}\right)^2 Q_4^\prime + \frac{n-1}{2}P_4\left(\sigma + \frac{n-1}{4}\sigma^2 + O\big((n-1)^2\big)\right) .
\end{align*}
Multiplying by $\frac{4}{(n-1)^2}$ and taking the limit $n\to 1$ yields~\eqref{eqn:qprime_operator_covariant}.  The invariance~\eqref{eqn:qprime_integral} then follows from the self-adjointness of $P_4^\prime$ and $P_4$ on their respective domains and the facts that $P_4(1)=0$ for any contact form and $P_4^\prime(1)=0$ for any pseudo-Einstein contact form.
\end{proof}

We conclude this section with a useful observation about the sign of the $P^\prime$-operator, which can be regarded as a CR analogue of a result of Gursky~\cite{Gursky1999} for the Paneitz operator in conformal geometry.

\begin{prop}
\label{prop:gursky}
Let $(M^3,J)$ be a compact CR manifold which admits a pseudo-Einstein contact form $\theta$ with nonnegative scalar curvature.  Then $P_4^\prime\geq 0$ and the kernel of $P^\prime$ consists of the constants.
\end{prop}

\begin{proof}

It follows from~\eqref{eqn:q_operator_covariant} that the conclusion $P_4^\prime\geq0$ with $\ker P_4^\prime=\bR$ is CR invariant, so we may compute in the scale $\theta$.  From the definition of the sublaplacian we see that
\[ \Delta_b^2 - 2\Imaginary\nabla^\beta\left(A_{\alpha\beta}\nabla^\alpha\right) = 2\Real\nabla^\alpha\left(\nabla_\alpha\nabla^\beta\nabla_\beta + P_\alpha\right) . \]
It thus follows that the $P^\prime$-operator is equivalently defined via the formula
\begin{equation}
\label{eqn:signed_version}
P_4^\prime u = 4\Real\nabla^\alpha\left(2\nabla_\alpha\nabla^\beta\nabla_\beta u - R\nabla_\alpha u\right)
\end{equation}
for all $u\in\mP$.  Multiplying~\eqref{eqn:signed_version} by $u$ and integrating yields
\[ \int_M u\,P_4^\prime u = 4\int_M \left(2\left|\nabla^\beta\nabla_\beta u\right|^2 + 2R\left|\nabla_\beta u\right|^2\right) . \]
Since $R\geq0$, this is clearly nonnegative, showing that $P_4^\prime\geq0$.  Moreover, if equality holds, then $\nabla^\beta\nabla_\beta u=0$, which is easily seen to imply that $u$ is constant, as desired.
\end{proof}

It would be preferable for Proposition~\ref{prop:gursky} to require checking only CR invariant assumptions.  For instance, one might hope to prove the same result assuming that $(M^3,J)$ has nonnegative CR Yamabe constant and admits a pseudo-Einstein contact form.  However, it is at present unclear whether these assumptions imply that one can choose a contact form as in the statement of Proposition~\ref{prop:gursky}.

%% file: covariance.tex
\section{CR covariance of the $P^\prime$-operator}
\label{sec:covariance}

In this section we give a direct computational proof of transformation formula~\eqref{eqn:q_operator_covariant} of the $P^\prime$-operator after a conformal change of contact form.  Indeed, we will compute the transformation formula for the operator $P_4^\prime$ as defined by~\eqref{eqn:q_crit} acting on functions --- rather than only pluriharmonic functions --- and thus establish that one cannot hope to find an invariant operator acting instead on the kernel of the CR Paneitz operator.

To begin, we recall from Lemma~\ref{lem:cr_change} and~\eqref{eqn:Walpha} that given a three-dimensional pseudohermitian manifold $(M^3,J,\theta)$ and an arbitrary one-form $\tau_\alpha\in\mE_\alpha(-1,-1)$, if $\hat\theta=e^\sigma\theta$ then
\begin{equation}
\label{eqn:hirachi}
\widehat{\nabla^\alpha}\widehat{\tau_\alpha} = \nabla^\alpha\tau_\alpha \quad\text{and}\quad \widehat{W_\alpha} = W_\alpha - 3P_\alpha\sigma .
\end{equation}

Another useful computation in preparation for our identification of the transformation law of $P_4^\prime$ is the following expression for the CR Paneitz operator applied to a product of two functions.

\begin{lem}
\label{lem:cr_prod}
Let $(M^3,J,\theta)$ be a pseudohermitian three-manifold and let $P_4$ be the CR Paneitz operator.  Given $f,\sigma\in\mE(0,0)$, it holds that
\begin{align*}
\frac{1}{4}P_4(f\sigma) & = \frac{1}{4}fP_4(\sigma) + \frac{1}{4}\sigma P_4(f) + 4\Real\left(\nabla^\alpha\sigma\nabla_\alpha\nabla^\beta\nabla_\beta f\right) + 4\Real\left(\nabla^\alpha f\nabla_\alpha\nabla^\beta\nabla_\beta\sigma\right) \\
& \quad + 2\Real\left(R\nabla^\alpha\sigma\nabla_\alpha f\right) + 2\Real\left(\nabla^\alpha\nabla^\beta\sigma\nabla_\alpha\nabla_\beta f\right) + 4\Real\left(\nabla^\alpha\nabla_\alpha\sigma\nabla_\beta\nabla^\beta f\right) .
\end{align*}
Alternatively,
\begin{align*}
P_4(\sigma f) & = \sigma P_4(f) + f P_4(\sigma) + 4\nabla^\alpha\sigma P_\alpha f + 4\nabla^\alpha f P_\alpha\sigma \\
& \quad + 4\nabla^\alpha\left(2\nabla_\alpha\sigma\nabla_\beta\nabla^\beta f + \nabla^\beta\sigma \nabla_\alpha\nabla_\beta f + 2\nabla_\alpha f\nabla_\beta\nabla^\beta\sigma + \nabla^\beta f \nabla_\alpha\nabla_\beta\sigma\right) .
\end{align*}
\end{lem}

\begin{proof}

The second expression follows by a direct expansion using the second formula for $P_4$ in~\eqref{eqn:subplacian_squared}.  To establish the first expression, observe that the term involving $\nabla^\alpha\sigma$ in the second expression of the lemma is
\begin{equation}
\label{eqn:intermediate}
\nabla^\alpha\sigma\left(\nabla_\alpha\nabla_\beta\nabla^\beta f + iA_{\alpha\beta}\nabla^\beta f + \nabla^\beta\nabla_\beta\nabla_\alpha f\right) .
\end{equation}
Using the assumption that $M$ is three-dimensional and the commutator formula
\[ \nabla_{\bar\gamma}\nabla_\beta\nabla_\alpha f - \nabla_\beta\nabla_{\bar\gamma}\nabla_\alpha f = i\nabla_0\nabla_\alpha f\, h_{\beta\bar\gamma} + R_\alpha{}^\rho{}_{\beta\bar\gamma}\nabla_\rho f \]
due to Lee~\cite[Lemma~2.3]{Lee1988}, it holds that
\begin{equation}
\label{eqn:lee_third_commutator}
\nabla^\beta\nabla_\beta\nabla_\alpha f - \nabla_\alpha\nabla^\beta\nabla_\beta f = i\nabla_0\nabla_\alpha f + R\nabla_\alpha f .
\end{equation}
It then follows from Lemma~\ref{lem:lee_commutator} that~\eqref{eqn:intermediate} can be rewritten
\[ \nabla^\alpha\sigma\left(2\nabla_\alpha\nabla^\beta\nabla_\beta f + R\nabla_\alpha f\right), \]
from which the desired expression immediately follows.
\end{proof}

The transformation formulae~\eqref{eqn:hirachi} immediately yield the transformation formulae for the zeroth and first order terms of $P_4^\prime$.  Thus it remains to compute the transformation formulae for the higher order terms.

\begin{prop}
\label{prop:D_covariant}
Let $(M^3,J,\theta)$ be a pseudohermitian three-manifold and define the operator $D\colon\mE(0,0)\to\mE(-2,-2)$ by
\[ Df = 4\Delta_b^2 f - 8\Imaginary\left(\nabla^\alpha(A_{\alpha\beta}\nabla^\beta f)\right) - 4\Real\left(\nabla^\alpha(R\nabla_\alpha f)\right) . \]
If $\hat\theta=e^\sigma\theta$ is another choice of contact form, then
\begin{align*}
\widehat{Df} & = Df + 8\Real\nabla^\alpha\left(2\nabla_\beta\nabla^\beta\sigma\nabla_\alpha f + \nabla^\beta\nabla_\beta\sigma\nabla_\alpha f + \nabla_\alpha\nabla_\beta f \nabla^\beta\sigma - \nabla^\beta\nabla_\beta f \nabla_\alpha\sigma\right) .
\end{align*}
\end{prop}

\begin{proof}

To begin, consider how each summand of $D$ transforms.  By a straightforward application of Lemma~\ref{lem:cr_change}, we compute that
\begin{align*}
\widehat{\Delta_b^2} \widehat{f} & = \Delta_b\left(\Delta_b f - \lp\nabla f,\nabla\sigma\rp\right) + 2\Real\nabla^\alpha\left((\Delta_b f - \lp\nabla f,\nabla\sigma\rp)\nabla_\alpha\sigma\right) \\
\widehat{\nabla^\alpha}\left(\widehat{A_{\alpha\beta}}\widehat{\nabla^\beta} \widehat{f}\right) & = \nabla^\alpha\left((A_{\alpha\beta} + i\nabla_\alpha\nabla_\beta\sigma - i\nabla_\alpha\sigma\nabla_\beta\sigma)\nabla^\beta f\right) \\
\widehat{\nabla^\alpha}\left(\widehat{R}\widehat{\nabla_\alpha}\widehat{f}\right) & = \nabla^\alpha\left((R+2\Delta_b\sigma - 2\lv\nabla_\gamma\sigma\rv^2)\nabla_\alpha f\right) .
\end{align*}
The conclusion of the proposition then follows immediately from a straightforward computation.
\end{proof}

Together, \eqref{eqn:hirachi}, Lemma~\ref{lem:cr_prod}, and Proposition~\ref{prop:D_covariant} yield another proof of the transformation law~\eqref{eqn:q_operator_covariant} of the $P^\prime$-operator.  In fact, the computations above allow us to compute the transformation rule for $P_4^\prime$ under a change of contact form when the local formula~\eqref{eqn:q_crit} is extended to all of $C^\infty(M)$.

\begin{prop}
Let $(M^3,J,\theta)$ be a pseudohermitian three-manifold and let $\sigma\in C^\infty(M)$.  Set $\hat\theta=e^\sigma\theta$ and denote by $\widehat{P_4^\prime}$ and $P_4^\prime$ the operator~\eqref{eqn:q_crit} defined in terms of $\hat\theta$ and $\theta$, respectively.  Then
\begin{equation}
\label{eqn:p4prime_crit_genl_transformation}
e^{2\sigma}\widehat{P_4^\prime}(f) = P_4^\prime(f) + P_4(f\sigma) - \sigma P_4(f) - 8\Real\left(P_\alpha f \nabla^\alpha \sigma\right)
\end{equation}
for all $f\in C^\infty(M)$.  In particular,
\[ e^{2\sigma}\widehat{P_4^\prime}(f) = P_4^\prime(f) + P_4(f\sigma) \]
for all $f\in\mP$.
\end{prop}

\begin{remark}

The transformation rule~\eqref{eqn:p4prime_crit_genl_transformation} obviously remains true when one adds a multiple of the CR Paneitz operator to $P_4^\prime$.
\end{remark}

\begin{proof}

It follows from~\eqref{eqn:hirachi} and Proposition~\ref{prop:D_covariant} that
\begin{align*}
e^{2\sigma}\widehat{P_4^\prime}(f) & = P_4^\prime(f) + f P_4\sigma - 8\Real\left(P_\alpha\sigma \nabla^\alpha f\right) \\
& \quad + 8\Real\nabla^\alpha\left(2(\nabla_\beta\nabla^\beta\sigma)(\nabla_\alpha f) + (\nabla^\beta\nabla_\beta\sigma)(\nabla_\alpha f) + (\nabla_\alpha\nabla_\beta f)(\nabla^\beta\sigma) - (\nabla^\beta\nabla_\beta f)(\nabla_\alpha\sigma)\right) .
\end{align*}
Using Lemma~\ref{lem:cr_prod} to write $fP_4\sigma$ in terms of $P_4(f\sigma)$, we find that
\begin{align*}
e^{2\sigma}\widehat{P_4^\prime}(f) & = P_4^\prime(f) + P_4(\sigma f) - \sigma P_4f - 4\Real\left(P_\alpha f\nabla^\alpha\sigma\right) \\
& \quad + 4\Real\left((2\nabla^\alpha\nabla_\beta\nabla^\beta\sigma - \nabla^\alpha\nabla^\beta\nabla_\beta\sigma - \nabla_\beta\nabla^\beta\nabla^\alpha\sigma + 3iA^{\alpha\beta}\nabla_\beta\sigma)\nabla_\alpha f\right) \\
& \quad - 4\Real\left((2\nabla^\alpha\nabla_\beta\nabla^\beta f + 2\nabla^\alpha\nabla^\beta\nabla_\beta f - \nabla_\beta\nabla^\beta\nabla^\alpha f)\nabla_\alpha\sigma\right) .
\end{align*}
The result then follows by using~\eqref{eqn:lee_third_commutator} to commute derivatives in the last two lines and the definition of the third order operator $P_\alpha$.
\end{proof}

%% file: qprime.tex
\section{CR transformation property of the $Q^\prime$-curvature}
\label{sec:qprime}

In this section we give a direct computational proof of the transformation formula~\eqref{eqn:qprime_operator_covariant} relating the $Q^\prime$-curvatures of two pseudo-Einstein contact forms on the same CR manifold.  As in Section~\ref{sec:covariance}, we will in fact compute how the scalar~\eqref{eqn:q4prime_crit} transforms under a conformal change of contact form without assuming either contact form is pseudo-Einstein.  This has two benefits.  First, it makes clear that the $Q^\prime$-curvature only transforms as in~\eqref{eqn:qprime_operator_covariant} when both contact forms are pseudo-Einstein, as opposed to having vanishing $Q$-curvature.  Second, it will allow us to prove Theorem~\ref{thm:qprime_upper_bound} by appealing to the resolution of the CR Yamabe Problem~\cite{ChengMalchiodiYang2013,JerisonLee1987,JerisonLee1988}.

First, as a consequence of Lemma~\ref{lem:cr_change}, we see that if $\hat\theta=e^\sigma\theta$, then
\begin{align*}
\widehat{R}^2 & = R^2 + 4R\Delta_b\sigma + 4(\Delta_b\sigma)^2 - 4R\lv\nabla_\gamma\sigma\rv^2 - 8\lv\nabla_\gamma\sigma\rv^2\Delta_b\sigma + 4\lv\nabla_\gamma\sigma\rv^4 \\
4\lv\widehat{A_{\alpha\beta}}\rv^2 & = 4\lv A_{\alpha\beta}\rv + 8\Imaginary\left(A_{\alpha\beta}\nabla^\alpha\nabla^\beta\sigma\right) + 4\lv\nabla_\alpha\nabla_\beta\sigma\rv^2 \\
& \quad - 8\Imaginary\left(A_{\alpha\beta}\nabla^\alpha\sigma\nabla^\beta\sigma\right) - 8\Real\left((\nabla_{\alpha\beta}\sigma)\nabla^\alpha\sigma\nabla^\beta\sigma\right) + 4\lv\nabla_\gamma\sigma\rv^4 \\
2\widehat{\Delta_b} \widehat{R} & = 2\Delta_b\left(P+2\Delta_b\sigma-2\lv\nabla_\gamma\sigma\rv^2\right) + 4\Real\nabla^\beta\left(\left(R+2\Delta_b\sigma-2\lv\nabla_\gamma\sigma\rv^2\right)\nabla_\beta\sigma\right) .
\end{align*}
It is immediately clear that the transformation law for $Q_4^\prime$ depends at most quadratically on $\sigma$.  Using the three-dimensional Bochner formula (cf.\ \cite{ChangChiu2007,ChangTieWu2010,ChanilloChiuYang2010,Chiu2006})
\begin{align*}
-\Delta_b\lv\nabla_\gamma\sigma\rv^2 & = 2\nabla_\alpha\nabla_\beta\sigma\nabla^\alpha\nabla^\beta\sigma + 2\nabla_\alpha\nabla^\alpha\sigma\nabla^\beta\nabla_\beta\sigma - \lp\nabla_b\sigma,\nabla_b\Delta_b\sigma\rp \\
& \quad - 2\Real\left(\nabla^\alpha\sigma(\nabla_\alpha\nabla_\beta\nabla^\beta\sigma-\nabla^\beta\nabla_\beta\nabla_\alpha\sigma)\right)
\end{align*}
together with the consequence
\begin{align*}
\frac{1}{2}P_4(\sigma^2) - \sigma P_4(\sigma) & = 8\Real\left(\nabla^\alpha\sigma P_\alpha\sigma\right) + 8\Real\left(\nabla^\alpha\sigma\nabla^\beta\nabla_\beta\nabla_\alpha\sigma\right) \\
& \quad + 4\nabla^\alpha\nabla^\beta\nabla_\alpha\nabla_\beta\sigma + 8\nabla_\alpha\nabla^\alpha\sigma\nabla^\beta\nabla_\beta\sigma - R\lv\nabla_\gamma\sigma\rv^2
\end{align*}
of Lemma~\ref{lem:cr_prod}, it follows immediately that the term of $\hat Q_4^\prime$ which is quadratic in $\sigma$ is given by
\[ U(\sigma) := \frac{1}{2}P_4(\sigma^2) - \sigma P_4(\sigma) - 16\Real\left(\nabla^\alpha\sigma P_\alpha\sigma\right) . \]
In particular, if $\sigma\in\mP$, then
\[ U(\sigma) = \frac{1}{2}P_4(\sigma^2) , \]
as expected.

On the other hand, the term of $\hat Q_4^\prime$ which is linear in $\sigma$ is given by
\begin{align*}
V(\sigma) & := 4\Delta_b^2\sigma - 8\Imaginary\left(\nabla^\alpha(A_{\alpha\beta}\nabla^\beta\sigma)\right) - 4\Real\left(\nabla^\alpha(R\nabla_\alpha\sigma)\right) + 8\Real\left(W_\alpha\nabla^\alpha\sigma\right) \\
& = P_4^\prime(\sigma) + \frac{16}{3}\Real\left(W_\alpha\nabla^\alpha\sigma\right) - Q\sigma \\
& = P_4^\prime(\sigma) + \frac{16}{3}\Real\nabla^\alpha\left(\sigma W_\alpha\right) + 3Q\sigma .
\end{align*}
In particular, if $\theta$ is a pseudo-Einstein contact form, then
\[ V(\sigma) = P_4^\prime(\sigma), \]
as expected.  In fact, we have computed the general transformation formula for the scalar invariant
\begin{equation}
\label{eqn:general_q4prime_defn}
Q_4^\prime = 2\Delta_b R - 4\lv A_{\alpha\beta}\rv^2 + R^2 .
\end{equation}

\begin{prop}
Let $(M^3,J,\theta)$ be a three-dimensional pseudohermitian manifold, regard $P_4^\prime$ as an operator $P_4^\prime\colon C^\infty(M)\to C^\infty(M)$, and define $Q_4^\prime$ by~\eqref{eqn:general_q4prime_defn}.  Given any $\sigma\in C^\infty(M)$, the scalars $Q_4^\prime$ and $\hat Q_4^\prime$ defined in terms of the contact forms $\theta$ and $\hat\theta=e^\sigma\theta$, respectively, are related by
\begin{equation}
\label{eqn:general_q4prime}
\begin{split}
e^{2\sigma}\hat Q_4^\prime & = Q_4^\prime + P_4^\prime(\sigma) + \frac{16}{3}\Real\nabla^\alpha\left(\sigma W_\alpha\right) + 3Q\sigma \\
& \quad + \frac{1}{2}P_4(\sigma^2) - \sigma P_4(\sigma) - 16\Real\left((\nabla^\alpha\sigma)(P_\alpha\sigma)\right) .
\end{split}
\end{equation}
In particular, if $M$ is compact, then
\begin{equation}
\label{eqn:general_integral_q4prime}
\int_M \hat Q_4^\prime\,\hat\theta\wedge d\hat\theta = \int_M Q_4^\prime\,\theta\wedge d\theta + 3\int_M \left(\sigma P_4\sigma + 2Q\sigma\right)\theta\wedge d\theta
\end{equation}
for $Q=P_4^\prime(1)$ Hirachi's $Q$-curvature~\eqref{eqn:Q_divergence}.
\end{prop}

\begin{proof}

\eqref{eqn:general_q4prime} follows from the computations given above.  \eqref{eqn:general_integral_q4prime} follows by integration by parts.
\end{proof}
%

\begin{proof}[Proof of Theorem~\ref{thm:qprime_upper_bound}]

Let $\hat\theta$ be a CR Yamabe contact form; that is, suppose that $\Vol_{\hat\theta}(M)=1$ and $R_{\hat\theta}=\Lambda[\theta]$ for $\Lambda[\theta]$ the CR Yamabe constant of $(M^3,J,\theta)$.  Then
\begin{equation}
\label{eqn:basic_sphere_step1}
\int_M \hat R^2\,\hat\theta\wedge d\hat\theta = \Lambda[\theta]^2 \leq \Lambda[S^3]^2
\end{equation}
for $\Lambda[S^3]=\Vol(S^3)$ the CR Yamabe constant of the standard CR three-sphere.  Moreover, by the CR Positive Mass Theorem~\cite{ChengMalchiodiYang2013}, equality holds in~\eqref{eqn:basic_sphere_step1} if and only if $(M^3,J,\theta)$ is CR equivalent to the standard CR three-sphere.  On the other hand, the nonnegativity of the CR Paneitz operator combined with~\eqref{eqn:general_integral_q4prime} yield
\begin{equation}
\label{eqn:basic_sphere_step2}
\int_M Q_4^\prime\,\theta\wedge d\theta \leq \int_M \hat Q_4^\prime\,\hat\theta\wedge d\hat\theta,
\end{equation}
while the expression~\eqref{eqn:general_q4prime_defn} yields
\begin{equation}
\label{eqn:basic_sphere_step3}
\int_M \hat Q_4^\prime\,\hat\theta\wedge d\hat\theta \leq \int_M \hat R^2\,\hat\theta\wedge d\hat\theta .
\end{equation}
The result then follows from~\eqref{eqn:basic_sphere_step1}, \eqref{eqn:basic_sphere_step2}, and~\eqref{eqn:basic_sphere_step3}.
\end{proof}

%% file: tractor.tex
\section{CR tractor bundles and the CR Paneitz operator}
\label{sec:tractor}

In this appendix we give the derivation of the CR Paneitz operator in general dimension using tractor bundles in CR geometry, as outlined by Gover and Graham~\cite{GoverGraham2005}.  In the interests of brevity, we compute in a fixed scale $\theta$ and only state the necessary tractor formulae, and refer the reader to~\cite{GoverGraham2005} for definitions of the tractor bundles and operators we use here.

The main objects we are concerned with are the CR tractor bundle $\mE_A\cong\mE(1,0)\oplus\mE_\alpha(1,0)\oplus\mE(0,-1)$, its canonical connection, and the tractor-$D$ operator $\bD\colon\mE^\ast(w,w^\prime)\to\mE_A\otimes\mE^\ast(w-1,w^\prime)$, which are given by
\begin{align*}
\nabla_\beta\begin{pmatrix}\sigma\\\tau_\alpha\\\rho\end{pmatrix} & = \begin{pmatrix}\nabla_\beta\sigma-\tau_\beta\\\nabla_\beta\tau_\alpha+i\sigma A_{\alpha\beta}\\\nabla_\beta\rho - P_\beta{}^\alpha\tau_\alpha + \sigma T_\beta\end{pmatrix} \\
\nabla_{\bar\beta}\begin{pmatrix}\sigma\\\tau_\alpha\\\rho\end{pmatrix} & = \begin{pmatrix}\nabla_{\bar\beta}\sigma\\\nabla_{\bar\beta}\tau_\alpha+\sigma P_{\alpha\bar\beta}+\rho h_{\alpha\bar\beta}\\\nabla_{\bar\beta}\rho+iA^\alpha{}_{\bar\beta}\tau_\alpha-\sigma T_{\bar\beta}\end{pmatrix} \\
\nabla_0\begin{pmatrix}\sigma\\\tau_\alpha\\\rho\end{pmatrix} & = \begin{pmatrix}\nabla_0\sigma+\frac{i}{n+2}P\sigma-i\rho\\\nabla_0\tau_\alpha-iP_\alpha{}^\beta\tau_\beta+\frac{i}{n+2}P\tau_\alpha+2i\sigma T_\alpha \\\nabla_0\rho + \frac{i}{n+2}P\rho + 2iT^\alpha\tau_\alpha+iS\sigma\end{pmatrix} \\
\bD_A f & = \begin{pmatrix} w(n+w+w^\prime)f\\(n+w+w^\prime)\nabla_\alpha f\\-\left(\nabla^\beta\nabla_\beta f + iw\nabla_0 f + w(1+\frac{w^\prime-w}{n+2})Pf\right)\end{pmatrix} ,
\end{align*}
where $\mE^\ast(w,w^\prime)$ denotes any (weighted) tractor bundle.  As always, the topmost nonvanishing slot is CR invariant.  In particular, we see that the bottom row of $\bD_A\bD_Bf$ is the topmost nonvanishing row if $n+w+w^\prime=1$; as is straightforward to check, if we assume that $f\in\mE(w,w^\prime)$ for $n+w+w^\prime=1$, then the ``bottom left'' spot is the only nonvanishing term, and hence is CR invariant.  More precisely, the operator $P_4$ defined by
\begin{equation}
\label{eqn:tractor_paneitz}
-\left(\nabla^\beta\nabla_\beta+i(w-1)\nabla_0+(w-1)(1+\frac{w^\prime-w+1}{n+2})P\right)\bD_A f = \begin{pmatrix} 0\\0\\\frac{1}{4}P_4f\end{pmatrix}
\end{equation}
will necessarily be a CR covariant operator which, as we shall see, has leading order term $\Delta_b^2+T^2$ (this is the reason for the factor of $4$).  To get the usual CR Paneitz operator, we need to assume further that $w=w^\prime$; in particular, $w=-\frac{n-1}{2}$.

In order to evaluate~\eqref{eqn:tractor_paneitz} to determine $P_4$, we need to know the bottom components of both $\nabla_0\bD_A f$ and $\nabla^\beta\nabla_\beta\bD_A f$.  The latter is the most involved computation: Marking irrelevant terms by asterisks, we see that
\begin{align*}
\nabla^\beta\nabla_\beta\begin{pmatrix}\sigma\\\tau_\alpha\\\rho\end{pmatrix} & = h^{\beta\bar\gamma}\nabla_{\bar\gamma}\nabla_\beta\begin{pmatrix}\sigma\\\tau_\alpha\\\rho\end{pmatrix} \\
& = h^{\beta\bar\gamma}\nabla_{\bar\gamma}\begin{pmatrix}\nabla_\beta\sigma-\tau_\beta\\\nabla_\beta\tau_\alpha+i\sigma A_{\alpha\beta}\\\nabla_\beta\rho - P_\beta^\alpha\tau_\alpha + \sigma T_\beta\end{pmatrix} \\
& = \begin{pmatrix}\ast\\\ast\\\nabla^\beta\nabla_\beta\rho - \nabla^\beta(P_\beta^\alpha\tau_\alpha - \sigma T_\beta) + iA^{\alpha\beta}\nabla_\beta\tau_\alpha - \sigma A^{\alpha\beta}A_{\alpha\beta}-(\nabla_\beta\sigma-\tau_\beta) T^\beta\end{pmatrix}
\end{align*}
In particular, we see that the bottom component of $\nabla^\beta\nabla_\beta\bD_Af$ is given by
\begin{equation}
\label{eqn:tractor_laplacian_special}
\begin{split}
& -\nabla^\beta\nabla_\beta(\nabla^\alpha\nabla_\alpha f+iw\nabla_0f+wPf)-\nabla^\beta\left(P_\beta^\alpha\nabla_\alpha f-wfT_\beta\right) \\
& \quad + iA^{\alpha\beta}\nabla_\beta\nabla_\alpha f - wf A^{\alpha\beta}A_{\alpha\beta} - (w-1)T^\beta\nabla_\beta f .
\end{split}
\end{equation}
The other derivative we must compute is $\nabla_0\bD_A f$; it is straightforward to check that the bottom component is given by
\begin{equation}
\label{eqn:tractor_0_special}
\begin{split}
& -\nabla_0\left(\nabla^\beta\nabla_\beta f + iw\nabla_0f + wPf\right) - \frac{i}{n+2}P\left(\nabla^\beta\nabla_\beta f + iw\nabla_0f + wPf\right) \\
& \quad + 2iT^\alpha\nabla_\alpha f+iwSf .
\end{split}
\end{equation}
Evaluating~\eqref{eqn:tractor_paneitz} using~\eqref{eqn:tractor_laplacian_special} and~\eqref{eqn:tractor_0_special}, we thus find that (after identifying tractor terms with their bottom components)
\begin{align*}
\frac{1}{4}P_4^\prime f & = -\nabla^\beta\nabla_\beta\bD_A f - i(w-1)\nabla_0\bD_A f - \frac{(w-1)(n+3)}{n+2}P\bD_A f \\
& = \nabla^\beta\nabla_\beta(\nabla^\alpha\nabla_\alpha f+iw\nabla_0f+wPf) + \nabla^\beta\left(P_\beta^\alpha\nabla_\alpha f-wfT_\beta\right) \\
& \quad - iA^{\alpha\beta}\nabla_\beta\nabla_\alpha f + wf A^{\alpha\beta}A_{\alpha\beta} + (w-1)T^\beta\nabla_\beta f \\
& \quad + i(w-1)\nabla_0\left(\nabla^\beta\nabla_\beta f + iw\nabla_0f + wPf\right) + (w-1)P\left(\nabla^\beta\nabla_\beta f + iw\nabla_0f + wPf\right) \\
& \quad + 2(w-1)T^\alpha\nabla_\alpha f + w(w-1)Sf .
\end{align*}

Our goal is now to simplify this so that we can identify the CR Paneitz operator.  Towards that end, let us regroup terms into those with a $w$ coefficient and those without; in other words, write
\begin{equation}
\label{eqn:lf_split}
P_4^\prime f = Af + wBf
\end{equation}
for
\begin{align*}
\frac{1}{4}Af & = \nabla^\beta\nabla_\beta\nabla^\alpha\nabla_\alpha f - i\nabla_0\nabla^\beta\nabla_\beta f \\
& \quad + \nabla^\beta\left(P_\beta^\alpha\nabla_\alpha f\right) - iA^{\alpha\beta}\nabla_\beta\nabla_\alpha f - 3T^\beta\nabla_\beta f - P\nabla^\beta\nabla_\beta f \\
\frac{1}{4}Bf & = -(w-1)\nabla_0\nabla_0 f + i\nabla_0\nabla^\beta\nabla_\beta f + i\nabla^\beta\nabla_\beta\nabla_0f \\
& \quad + \nabla^\beta\nabla_\beta(Pf) - \nabla^\beta(T_\beta f) + A^{\alpha\beta}A_{\alpha\beta} f + 3T^\beta\nabla_\beta f + i(w-1)\nabla_0(Pf) \\
& \quad + P\nabla^\beta\nabla_\beta f + i(w-1)P\nabla_0f + (w-1)P^2f + (w-1)Sf .
\end{align*}

First, let us rewrite $Af$ in a more familiar way.  Using~\eqref{eqn:schouten_bianchi} and~\eqref{eqn:subplacian_squared}, it is straightforward to check that
\begin{align*}
\frac{1}{4}Af & = \frac{1}{4}Cf + i(n-1)\nabla_0\nabla^\beta\nabla_\beta f + i(n-1)A^{\alpha\beta}\nabla_\beta\nabla_\alpha f + in\left(\nabla_\beta A^{\alpha\beta}\right)\nabla_\alpha f \\
& \quad + (\nabla^\beta P+(n-1)T^\beta)\nabla_\beta f + (P^{\alpha\bar\beta}-Ph^{\alpha\bar\beta})\nabla_{\bar\beta}\nabla_\alpha f - 3T^\beta\nabla_\beta f \\
& = \frac{1}{4}Cf + \tracefree{P}{}^{\alpha\bar\beta}\nabla_{\bar\beta}\nabla_\alpha f \\
& \quad + \frac{n-1}{2}\bigg[2i\nabla_0\nabla^\beta\nabla_\beta f +2i\nabla_\beta(A^{\alpha\beta}\nabla_\alpha f) - \frac{2}{n}P\nabla^\beta\nabla_\beta f + 4T^\beta\nabla_\beta f\bigg] .
\end{align*}
Since $\tracefree{P_{\alpha\bar\beta}}=0$ and $w=0$ when $n=1$, we check in particular that $P_4f=Cf$ in this dimension.

Second, recalling that $w=-\frac{n-1}{2}$, we see from the above that
\begin{align*}
Af & = Cf + 4\tracefree{P}{}^{\alpha\bar\beta}\nabla_{\bar\beta}\nabla_\alpha f + wEf \\
\frac{1}{4}Ef & := -2i\nabla_0\nabla^\beta\nabla_\beta f - 2i\nabla_\beta(A^{\alpha\beta}\nabla_\alpha f) + \frac{2}{n}P\nabla^\beta\nabla_\beta f - 4T^\beta\nabla_\beta f .
\end{align*}
In particular, the operator $F$ defined by $F=B+E$ is such that $P_4f=Cf+4\tracefree{P}{}^{\alpha\bar\beta}\nabla_{\bar\beta}\nabla_\alpha f + wFf$, and is given by
\begin{align*}
\frac{1}{4}Ff & = (1-w)\nabla_0\nabla_0 f + i\nabla^\beta\nabla_\beta\nabla_0 f - i\nabla_0\nabla^\beta\nabla_\beta f - 2i\nabla_\beta(A^{\alpha\beta}\nabla_\alpha f) \\
& \quad + \frac{2(n+1)}{n}P\nabla^\beta\nabla_\beta f + 2i(w-1)P\nabla_0 f + \nabla^\beta P \nabla_\beta f + \nabla_\beta P \nabla^\beta f - T^\beta\nabla_\beta f - T_\beta\nabla^\beta f \\
& \quad + \left(\nabla^\beta\nabla_\beta P - \nabla^\beta T_\beta + i(w-1)\nabla_0 P + A^{\alpha\beta}A_{\alpha\beta} + (w-1)P^2 + (w-1)S\right)f \\
& = (1-w)\nabla_0\nabla_0f + i\nabla^\beta(A_{\alpha\beta}\nabla^\alpha f) - i\nabla_\beta(A^{\alpha\beta}\nabla_\alpha f) \\
& \quad + \frac{2(n+1)}{n}P\nabla^\beta\nabla_\beta f + 2i(w-1)P\nabla_0 f + \nabla^\beta P \nabla_\beta f + \nabla_\beta P \nabla^\beta f - T^\beta\nabla_\beta f - T_\beta\nabla^\beta f \\
& \quad + \left(\nabla^\beta\nabla_\beta P - \nabla^\beta T_\beta + i(w-1)\nabla_0 P + A^{\alpha\beta}A_{\alpha\beta} + (w-1)P^2 + (w-1)S\right)f .
\end{align*}

Writing this entirely in terms of $n$, $P_{\alpha\bar\beta}$, $P$, and $A_{\alpha\beta}$ then yields the desired form.

%% file: fefferman.tex
\section{Checking Via the Fefferman Metric}
\label{sec:fefferman}

In this appendix, we follow the other perspective of Gover and Graham~\cite{GoverGraham2005} and give the formula for the CR Paneitz operator using the Fefferman metric.  To arrive at the formula given in Definition~\ref{defn:cr_paneitz}, we use Lee's intrinsic formulation~\cite{Lee1986} of the Fefferman metric.

To begin, let $(M^{2n+1},J,\theta)$ be a pseudohermitian manifold and let $(\tilde M^{2n+2},g)$ be the Fefferman bundle, which is an $S^1$-bundle over $M$ with $g$ a particular Lorentzian metric.  The Paneitz operator $L_4$ on a pseudo-Riemannian manifold is defined by
\[ L_4u = \Delta^2 u + 4P^{ij}\nabla_i\nabla_j u - (N-2)P_i^i\Delta u - (N-6)(\nabla^jP_i^i)(\nabla_ju) + \frac{N-4}{2}Qu, \]
where $N=2n+2$ is the dimension of $\tilde M$, $P_{ij}=\frac{1}{N-2}\left(R_{ij} - \frac{1}{2(N-1)}R_k^kg_{ij}\right)$ is the Schouten tensor of $g$, $\Delta=\nabla^i\nabla_i$ is the Laplacian (with nonpositive spectrum), and
\[ Q = -\Delta P_i^i - 2P_{ij}P^{ij} + \frac{N}{2}\left(P_i^i\right)^2 \]
is the (conformal) $Q$-curvature.  The key facts about the Paneitz operator on the Fefferman bundle are that it is conformally invariant and that its restriction to functions which are invariant under the circle action is itself invariant under the circle action.  In particular, these facts together imply that $L_4$ descends to a CR covariant operator on $(M^{2n+1},J,\theta)$.  Explicitly, the operator $P_4$ defined by
\begin{equation}
\label{eqn:fefferman_paneitz}
P_4u = \frac{1}{4}\pi_\ast\left(L_4(\pi^\ast u)\right)
\end{equation}
will necessarily be a CR covariant operator of the form $\Delta_b^2+T^2$ plus lower order terms.  Its explicit form can be computed using the following sequence of lemmas which are a consequence of Lee's intrinsic characterization~\cite{Lee1986} of the Fefferman bundle.  First, we have the following simple expressions for the scalar curvature, the Laplacian, and the inner product of two gradients on both manifolds.

\begin{lem}
Let $(M^{2n+1},J,\theta)$ be a pseudohermitian manifold and let $(\tilde M^{2n+2},g)$ denote the associated Fefferman bundle.  Then, given any $u,v\in C^\infty(M)$, it holds that
\[ \pi_\ast\left(\Delta (\pi^\ast u)\right) = -2\Delta_b u , \quad \pi_\ast J = 2P, \quad \pi_\ast\lp\nabla(\pi^\ast u),\nabla(\pi^\ast v)\rp = 4\Real\left(\nabla^\alpha u\nabla_\alpha v\right) . \]
\end{lem}

Next, we have the relationship between the norms of the Schouten tensor on both manifolds.

\begin{lem}
Let $(M^{2n+1},J,\theta)$ be a pseudohermitian manifold and let $(\tilde M^{2n+2},g)$ denote the associated Fefferman bundle.  It holds that
\begin{align*}
\pi_\ast\left(P_{ij}P^{ij}\right) & = \frac{2(n+1)}{n}P_{\alpha\bar\beta}P^{\alpha\bar\beta} + \frac{2(n-1)}{n}A_{\alpha\beta}A^{\alpha\beta} \\
& \quad + \frac{4}{n(n+2)}\Imaginary\left(\nabla^\alpha\nabla^\beta A_{\alpha\beta}\right) + \frac{4}{n(n+2)}\Real\left(\nabla^\alpha\nabla_\alpha P\right) .
\end{align*}
\end{lem}

The last ingredient from~\cite{Lee1986} is the inner product of the Schouten tensor with a Hessian, which follows from the formulae for the Ricci tensor and the connection on both manifolds.

\begin{lem}
Let $(M^{2n+1},J,\theta)$ be a pseudohermitian manifold and let $(\tilde M^{2n+2},g)$ denote the associated Fefferman bundle.  Then, given any $u\in C^\infty(M)$, it holds that
\[ \pi_\ast\left(P^{ij}\nabla_i\nabla_ju\right) = 4\nabla_0^2 u - 16\Imaginary\left(A_{\alpha\beta}\nabla^\alpha\nabla^\beta u\right) + 16\Real\left(P^{\alpha\bar\beta}\nabla_{\bar\beta\alpha}^2u\right) - 48\Real\left(T_\alpha\nabla^\alpha u\right) . \]
\end{lem}

Putting these together, we provide another derivation for the formula given in Definition~\ref{defn:cr_paneitz} for the CR Paneitz operator.

\begin{prop}
Let $(M^{2n+1},J,\theta)$ be a pseudohermitian manifold and let $(\tilde M^{2n+2},g)$ denote the associated Fefferman bundle.  Denote by $F$ and $Q$ the operators
\begin{align*}
F(u) & = 4P^{ij}\nabla_i\nabla_ju - (N-2)P_k^k\Delta u - (N-6)(\nabla^jP_i^i)(\nabla_ju) \\
Q & = -\Delta P_i^i - 2P_{ij}P^{ij} + \frac{N}{2}(P_i^i)^2
\end{align*}
on $\tilde M^N$.  Then, given any $u\in C^\infty(M)$, it holds that
\begin{align*}
\frac{1}{4}\pi_\ast\left(\Delta^2 (\pi^\ast u)\right) & = \Delta_b^2 u \\
\frac{1}{4}\pi_\ast\left(F(\pi^\ast u)\right) & = \nabla_0^2 u - 4\Imaginary\nabla^\alpha\left(A_{\alpha\beta}\nabla^\beta u\right) + 4\tracefree{P}{}^{\alpha\bar\beta}\nabla_{\bar\beta\alpha}^2u - \frac{4(n^2-1)}{n}\Real\left(\nabla^\alpha(P\nabla_\alpha u)\right) \\
& \quad - \frac{32(n^2-1)}{n(n+2)}\Real\left((\nabla_\alpha P - \frac{in}{2(n+1)}\nabla^\beta A_{\alpha\beta})\nabla^\alpha u\right) \\
\frac{1}{4}\pi_\ast Q & = \frac{(n+1)^2}{n(n+2)}\Delta_b P - \frac{2}{n(n+2)}\Imaginary\left(\nabla^\alpha\nabla^\beta A_{\alpha\beta}\right) \\
& \quad - \frac{n+1}{n}\lv \tracefree{P}{}^{\alpha\bar\beta}\rv^2 - \frac{n-1}{n}\lv A\rv^2 + \frac{(n-1)(n+1)^2}{n^2} P^2 .
\end{align*}
In particular, Definition~\ref{defn:cr_paneitz} for the CR Paneitz operator agrees with the definition via~\eqref{eqn:fefferman_paneitz}.
\end{prop}